\documentclass[11pt]{amsart}
\usepackage{amsmath,amssymb,amscd,amsthm}
\usepackage{amsthm}
\usepackage[margin=1.3in]{geometry}
\usepackage[pagewise]{lineno}
\usepackage{setspace}

\usepackage{tikz-cd}
\usetikzlibrary{babel}
\usetikzlibrary{decorations.pathmorphing}
\usepackage[margin=1.3in]{geometry}
\usepackage{hyperref}
\hypersetup{hidelinks, colorlinks=true, allcolors=blue, pdfstartview=Fit, breaklinks=true}

\usepackage[all]{xy}

\usepackage{enumerate}
\usepackage{paralist}
\usepackage{amscd}

\newtheorem{thm}{Theorem}[subsection]
\newtheorem{lem}[thm]{Lemma}
\newtheorem{prop}[thm]{Proposition}

\newtheorem{cor}[thm]{Corollary}

\newtheorem*{prob}{\bf Problem}
\theoremstyle{definition}\newtheorem{df}[thm]{Definition}
\theoremstyle{definition}\newtheorem{rem}[thm]{Remark}
\newtheorem{exm}[thm]{\it Example}
\newtheorem{qn}[thm]{\bf Question}
\newtheorem{eq}[thm]{\it Example}
\theoremstyle{definition}

\renewcommand{\phi}{\varphi}

\newcommand{\N}{\mathbb{N}}
\newcommand{\Homeo}{\mathcal{H}}

\newcommand{\A}{\mathcal{A}}
\newcommand{\Z}{\mathbb{Z}}
\newcommand{\J}{\mathcal{J}}
\newcommand{\tildeX}{\widetilde{X}}
\newcommand{\X}{\underline{X}}

\newcommand{\KX}{{}_{k_2}X_{l_2}}
\newcommand{\F}{\tilde{F}}
\newcommand{\kX}{{}_{k_1}X_{l_1}}
\newcommand{\Kx}{{}_{k_2}[x]_{l_2}}
\newcommand{\kx}{{}_{k_1}[x]_{l_1}}
\newcommand{\jf}{\mathfrak{j}}
\newcommand{\I}{\mathcal{I}}
\newcommand{\D}{\mathcal{D}}
\newcommand{\Sph}{\mathbb{S}}
\newcommand{\La}{\mathcal{L}}
\newcommand{\R}{\mathbb{R}}
\newcommand{\Raf}{\mathcal{R}_0(\tilde{\alpha})}
\newcommand{\Rbt}{\mathcal{R}_0(\tilde{\beta})}
\newcommand{\C}{\mathbb{C}}
\newcommand{\T}{\mathbb{T}}
\newcommand{\taf}{\tilde{\alpha}}
\newcommand{\tbt}{\tilde{\beta}}

\newcommand{\Aff}{\operatorname{Aff}}
\newcommand{\Pj}{\mathcal{P}}

\newcommand{\id}{\operatorname{id}}

\newcommand{\ep}{\varepsilon}

\newcommand{\Lip}{\mathcal{L}}
\newcommand{\G}{\mathcal{G}}

\newcommand{\af}{{\alpha}}
\newcommand{\bt}{{\beta}}

\newcommand{\beq}{\begin{eqnarray}}
\newcommand{\eneq}{\end{eqnarray}}

\begin{document}
\title{A generalization of topological Rokhlin dimension and an embedding result}

\author{Zhuofeng He, Sihan Wei}

\maketitle




\begin{abstract} 
We generalize Gabor's notion of topological Rokhlin dimension of $\mathbb{Z}^k$-actions on compact metric space to a class of general discrete countable amenable group actions which involves the approximate subgroup structure. Then with this generalization, we conclude the finiteness of topological Rokhlin dimension, amenability dimension, dynamic asymptotic dimension and also of the nuclear dimension of the crossed product. An embedding result is also obtained, regarding those systems with mean dimension less that $m/2$ and with a finite-dimensional free factor.

\quad\par
\noindent{\bf Keywords}. Approximate groups $\cdot $ Topological Rokhlin dimension $\cdot$ marker property $\cdot$ amenability dimension $\cdot$ nuclear dimension $\cdot$ embedding problem
\quad\par
\quad\par
\noindent{\bf Mathematics Subject Classification (2020)} 46L05, 46L35, 37B05 (Primary)
\end{abstract}

\section{Introduction}
In the early 21st century, a radical idea was born, claiming that instead of considering “strict mathematical structures”, such as strict equalities, strict structures of algebraic objects(groups, rings, fields, algebras, etc.), it is somehow more appealing to look at the “approximate” ones. Here by the word “strict”, we mean an exact equality, or either an algebraic operation so that under which the element produced by any two elements in the underlying set still lies in it. The “approximate” version makes the strictness collapse. The algebraic structures, or in other words, sets with operations, in which the operations themselves do not produce elements in the underlying sets, but rather elements of larger sets which are in some sense “close” to original ones, were studied. Among these new structures, the theory of approximate groups is one of the most interesting objects, which plays a central role in modern additive combinatorics, the theory of mathematical quasi-crystals, theory of locally compact groups, to name of few(see the discussion historical perspectives of Sect. 1.1, \cite{CHT}).

To put it in a nutshell, an {\it approximate subgroup} in a group $G$ is a symmetric subset $\Lambda\subset G$ containing the identity, with the property that $A^2$ is covered by finitely many left(or right) translates of $A$ itself. This is close the the so-called {\it small doubling property}, which is defined to be such that $|A^2|\le K|A|$ for some constant $K$. Note that a results in \cite{BGT} states that if $G$ is finitely generated, and in addition, there is a $K>0$ and an increasing sequence $\{A_n\}$ of finite subsets of $G$ whose union is $G$ such that $|A_n^2|\le |A_n|$ for all $n\ge1$, $G$ is {\it virtually nilpotent}, and hence of {\it polynomial growth} equivalently. In the present paper, we use the following definition, which removes the symmetry condition.
\begin{df}[=Defintion \ref{AG}]
Let $G$ be a group and $A\Subset G$ be a finite subset. We call $A$ is a {\it (weakly) approximate (sub)group} in $G$, if 

(i) The group identity $e_G\in A$;

(ii) There is a natural number $L_A>0$, such that $A^{-1}A$ is covered by a union of $L_A$ left translates of $A$ by elements in $G$, i.e., there are group elements $g_1, g_2,\cdots, g_{L_A}\in G$
(depending on $A$ itself) with $A^{-1}A\subset \bigcup_{1\le i\le L_A} g_i\cdot A$.
\end{df}

For a (non-necessarily finite generated) group $G$, we say it has the {\it weakly approximate F$\o$lner condition(WAFC)}, if there is a F$\o$lner sequence $\{F_n\}\subset G$ such that every $F_n$ is an approximate subgroup with uniformly bounded $L_G=\sup_{n\ge1}L_{F_n}$, see Definition \ref{AFC}. Note that such constant $L_G$ depends on the choice of the F$\o$lner sequence. As is aforementioned, for a finitely generated group $G$, (WAFC) implies virtual nilpotence(we actually believe that this is an if and only if statement).

In the realm of measure-theoretic dynamical systems, the {\it Rokhlin Lemma}, a.k.a. the {\it Kakutani-Rokhlin Lemma}, has an extensive impact in ergodic theory. It states that every aperiodic measure-preserving dynamical system $X$ can be decomposed into an arbitrarily high tower $\mathcal{T}$ with its levels measurable, such that the remainder set $X\setminus \bigcup_{E\in\mathcal{T}}E$ is of arbitrarily small measure. It is such a massive impact so that the Rokhlin Lemma derives numerous Rokhlin-type concepts, variations and constructions in other mathematical branches such as different versions of {\it topological Rokhlin properties} in topological dynamical systems(see \cite{Gut2}--\cite{GT1} for instance), the {\it Rokhlin dimension, tracial Rokhlin dimension} or the {\it cyclic Rokhlin dimension} in $C^*$-dynamics and the classification program of simple amenable separable $C*$-algebras, see \cite{AGG},\cite{EN}--\cite{GHS},\cite{HWZ},\cite{K}--\cite{SWZ} for a (far from complete) list of them. Intuitively speaking, the key advantage stemming from these “Rokhlin type” properties or derived concepts is that, it reflects some “shift-type” behaviour of the underlying actions approximately, as is pointed out in \cite{SWZ}. In particular, the Rokhlin dimension among these, can be regarded as a “high dimensional” Rokhlin property, which is also what we will be focusing on in the present paper, plays a central role in the classification of $C^*$-dynamics in recent years.

Passing from topological dynamics to the classification program of $C^*$-algebras raised and developed by G. Elliott and his many collaborators, a special class of separable simple amenable $C^*$-algebras, called the {\it crossed products of unital $C^*$-algebras by countable amenable groups} have attracted substantial attention. The class of crossed products arising from minimal topological dynamical systems on  compact metric spaces is, among separable simple $C^*$-algebras, one of the most natural to consider. On the other hand, similarly to topological dynamics, it turns out that many interesting examples of finite groups or $\Z$  on $C^*$-algebras have Rokhlin-type properties. Enlightened by these examples and the fact that unlike the situation of Von Neumann algebras, the lack of projections is ubiquitous for $C^*$-algebras, the Rokhlin dimension of $\Z$ and finite groups on $C^*$-algebras is defined and developed in \cite{HWZ}, by allowing towers consisting of positive elements instead of requiring projections. This concept was later generalized to actions of residually finite groups and finitely generated nilpotent groups in \cite{SWZ}. All these attempts consequently led to classification theorems. 

Back to crossed products arising from topological dynamical systems, in order to have an estimate of the nuclear dimension of $C^*$-crossed products, the {\it (topological) Rokhlin dimension} of $\Z^m$-actions (in fact of $\Z^m\times G$ for any locally finite group $G$) is defined in \cite{Sza} for topological dynamical systems. It is somehow surprising in \cite{Sza} that , the topological Rokhlin dimension is related to the {\it marker property}, which is a concept originally defined by Y. Gutman when dealing with the well-known {\it embedding problem} in topological dynamics. However, it is at the same time reasonable since the so-called marker property actually stems from the topological Rokhlin Lemma of \cite{Ln}. Therefore, G. Szab$\acute{\rm o}$ in \cite{Sza} asked following two questions.
\begin{qn}[Question 5.6, \cite{Sza}]
What is the right notion of Rokhlin dimension for actions of a larger class of countable, non-abelian, discrete, amenable groups?
\end{qn}

\begin{qn}[Question 5.7, \cite{Sza}]
Suppose that one can successfully generalize Rokhlin dimension to actions of a larger class of countable, discrete, amenable groups. For which groups is it automatic that free actions on finite dimensional spaces have finite Rokhlin dimension?
\end{qn}

We remark that, although a right notion of Rokhlin dimension for general group actions was absent, with a topological version of the {\it Ornstein-Weiss tower decomposition} called {\it almost finiteness} defined by D. Kerr which can be regarded as a dynamical $\mathcal{Z}$-stability, the celebrated theorem was obtained in \cite{KS}, stating that the crossed products of free minimal actions $G\curvearrowright X$ where $X$ is a compact metric space with finite covering dimension and $G$ is such that each of its finitely generated subgroups has subexponential growth are classified by the Elliott invariant (ordered $K$-theory paired with tracial states) and are simple ASH-algebras of topological dimension at most $2$.

One of the motives of the present paper is, to some extent, attempting to compensate the lack of such a right notion of Rokhlin dimension, especially on {\it discrete Heisenberg groups}. Recall that a countable group $G$ is said to have (WAFC), if it admits a F$\o$lner sequence $\{F_n\}$, each of which is an approximate subgroup with a uniform bounded sequence $\{L_{F_n}\}$. All the Heisenberg groups have (WAFC). Whatsoever, we note that we just don't necessarily require the group $G$ to be finitely generated.
\begin{df}[=Definition \ref{Rokdim}+Definition \ref{srd}]
Let $D\in\N\cup\{0\}$. We say that $(X,G)$ has {\it topological Rokhlin dimension} $D$, and denote ${\rm dim}_{\rm Rok}(X,G)=D$, if $D$ is the smallest natural number with the following property:

For every $n\in\N$, there exist open subsets $U_0,U_1,\cdots,U_D\subset X$ such that

\noindent (i) for every $0\le i\le D$, $\{g\overline{U_i}: g\in F_n\}$ is a disjoint (closed) tower with the shape $F_n$, and

\noindent (ii) the union of these (open) towers covers $X$, that is, $\bigcup_{0\le i\le D} F_nU_i=X$.

Furthermore, if there exists a natural number $N\ge n$ such that $F_n^2\subset F_N$ and the open sets $U_0,U_1,\cdots,U_D$ can be chosen with the property that the family $\{g\overline{U_i}: g\in F_N\}$ is disjoint for every $i=0,1,\cdots,D$, then we say $(X,G)$ has {\it strong Rokhlin dimension} $D$, written as ${\rm dim}_{\rm sRok}(X,G)=D$.
\end{df}
With the notion of topological (strong) Rokhlin dimension of groups satisfying (WAFC) and under a mild condition  which requires the group $G$ to be “not that nonabelian”, we get the following theorem.
\vspace{0.2cm}

\noindent {\bf Theorem A}\ (=Corollary \ref{frd}) Let $(X, G)$ be a topological dynamical system, where $X$ is a compact metric space and $G$ a countably infinite group with $|Z(F)|=\infty$ for every finite subset $F\subset G$. Let $\{F_n\}\subset G$ be a  symmetric F$\o$lner sequence such that $(G, \{F_n\})$ satisfies (WAFC) with respect to $L_G$.

If $(X,G)$ is free and ${\rm dim}(X)<\infty$, then $(X,G)$ has finite topological Rokhlin dimension. In particular,
\[{\rm dim}_{\rm Rok}(X,G)\le L_G({\rm dim}(X)+1)-1.\]
\vspace{0.2cm}

In Theorem A, we can replace the condition that $(X,G)$ is free with $(X,G)$ having an at most zero dimensional periodic point set, see \cite{Sza}. It is also worth-mentioning that, by our procedure, we can even arrange so that $(X,G)$  has finite strong topological Rokhlin dimension.
\vspace{0.2cm}

\noindent {\bf Theorem B}\ (=Corollary \ref{am}+Corollary \ref{eqf}) Keep the same assumptions on $G$ as in Theorem A and let $(X,G)$ be such that ${\rm dim}(X)<\infty$. Then $(X,G)$ is free if and only if $(X,G)$ has finite strong topological Rokhlin dimension.
\vspace{0.2cm}

As corollaries, these results (including what follows) apply to the case of $\mathbb{H}_{2n+1}(\Z)\times G$, where $\mathbb{H}_{2n+1}(\Z)$ is the discrete Heisenberg group, and $G$ is any of a finite group, a finitely generated abelian group, a locally finite group or a general countable group satisfying (WAFC) with a symmetric F$\o$lner sequence, to name a few.

The notions of{\it amenability dimension}, {\it dynamic asymptotic dimension}, {\it tower dimension}, {\it fine tower dimension} of group actions on topological spaces, {\it amenability dimension} of group actions on $C^*$-algebras and nuclear dimension of $C^*$-algebras were purposed and developed successively by several authors during the study of classification programs, including E. Guentner, R. Willett and G. Yu in \cite{GWY} for the dynamics asymptotic dimension and amenability dimension, D. Kerr and G. Szab$\acute{\rm o}$ in \cite{KS} for the tower dimension and fine tower dimension, W. Winter and J. Zacharias in \cite{WZ} for the nuclear diemsnion. In particular, G. Szab$\acute{\rm o}$, J. Wu and J. Zacharias estimated the amenability dimension of free actions by finitely generated, infinite, nilpotent groups on compact metric spaces of finite covering dimension, which gives the finiteness of the nuclear dimension of the $C^*$-crossed products.

By our finiteness theorem of the strong topological Rokhlin dimension of free group actions on compact metric spaces of finite covering dimension or more generally, with the condition $(TSBP\le d)$ for some $d\in\N\cap\{0\}$, we extend the result of amenability dimension to our setting.
\vspace{0.2cm}

\noindent {\bf Theorem C}\ (=Theorem \ref{fad})  Let $(X, G)$ be a topological dynamical system, where $X$ is a compact metric space and $G$ a countably infinite group with $|Z(F)|=\infty$ for every finite subset $F\subset G$. Let $\{F_n\}\subset G$ be a  symmetric F$\o$lner sequence such that $(G, \{F_n\})$ satisfies (WAFC) with respect to $L_G$.

If $(X, G)$ is free and ${\rm dim}(X)\le d$ for some $d\in\N\cup\{0\}$, then the induced $C^*$-algebraic action $\af: G\curvearrowright C(X)$ has finite amenability dimension estimated by the following inequality
\[{\rm dim}^{+1}_{\rm am}(\af)\le L_G(d+1).\]
\vspace{0.2cm}

As a corollary applying Theorem 8.6, Theorem 4.11 in \cite{GWY} and Theorem 5.14 in \cite{KS}, we obtain
\vspace{0.2cm}

\noindent {\bf Theorem D}\ (=Corollary \ref{nucdim}) Keep the same assumptions as in Theorem C on the group $G$ and the topological dynamical system $(X,G)$. Then the dynamic asymptotic dimension ${\rm dad}(X,G)$, tower dimension ${\rm dim}_{\rm tow}$ and fine tower dimension ${\rm dim}_{\rm ftow}$ of $(X, G)$ are all finite, bounded by
\[{\rm dad}^{+1}(X,G)\le L_G(d+1),\ \ {\rm dim}^{+1}_{\rm tow}(X,G)\le{\rm dim}^{+1}_{\rm ftow}(X,G)\le L_G(d+1)^2.\]
Consequently, the crossed product $C(X)\rtimes_{\af}G$ has finite nuclear dimension given by
\[{\rm dim}^{+1}_{\rm nuc}(C(X)\rtimes_{\af}G)\le L_G(d+1)^2.\]
If in addition, the action is minimal, then $C(X)\rtimes_\af G$ is $\mathcal{Z}$-stable, and also classifiable. In particular, these are all the cases if $X$ has finite covering dimension.
\vspace{0.2cm}

In topological dimension theory, it is a classic result that any compact metric space embeds in the {\it Hilbert cube} $\mathbb{K}=[0,1]^{\N}$, and if ${\rm dim}(X)<\infty$, then it embeds in the finite-dimensional cube $[0,1]^{2{\rm dim}(X)+1}$. This will follow that for any finite dimensional compact metric space $X$ and any homeomorphism $T: X\to X$, the dynamical system $(X,T)$ embeds in the shift $(([0,1]^n)^\Z, \sigma)$, where $n=2{\rm dim}(X)+1$. By a shift on $\Omega^\Z$ here where $\Omega$ is some compact metric space, we mean a shift operation $\sigma$ sending $x$ to $\sigma(x)$ with $(\sigma(x))_{n}=x_{n+1}$. Therefore, a natural question arose from it is to determine the smallest natural number $n$ so that $(X,T)$ can be embedded into the shift $([0,1]^n)^\Z,\sigma)$. The famous {\it Jaworski's embedding theorem} shows that for a topological dynamical system on a finite dimensional space, periodic points are the only obstruction for such an embedding. To be specific, Jaworski in 1974 showed that, if $(X,T)$ is a finite dimensional aperiodic system, then it embeds into the shift on $[0,1]^\Z$, see \cite{Ja}. Later by several topological dynamicists including E. Lindenstrauss(\cite{Ln}), Y. Gutman, M. Tsukamoto and Y. Qiao(\cite{Gut2}--\cite{GT2}) and many other collaborators of them, the embedding theorem was promoted to an extensive generalization. In particular, Y. Gutman and M. Tsukamoto proved in \cite{GT2} that, for minimal dynamical systems on compact metric spaces, the {\it mean dimension} is the only obstruction for the embedding problem. They show the celebrated theorem, claiming that any minimal dynamical system with ${\rm mdim}(X,T)<N/2$ embeds into the shift on $([0,1]^N)^\Z$.

The notion of mean dimension is originally defined by E. Lindenstrauss and B. Weiss, attempting to solve Auslender' problem asking that can all the minimal systems be embedded into $[0,1]^\Z$. Intuitively speaking, the mean dimension counts the number of parameters of systems per second, while the {\it topological entropy} counts the number of bits per second for describing the systems. Any finite dimensional dynamical system has mean dimension zero, so does any dynamical system with finite entropy or {\it unique ergodicity}. Therefore, mean dimension can be regarded as a {\it dual concept} of the topological entropy, mainly used to deal with infinite dimensional systems. On the other hand, it is an interesting and promising approach to the embedding problem in \cite{GQS}, namely, by applying the topological Rokhlin dimension defined by G. Szab$\acute{\rm o}$. In particular, they have the following result.
\begin{thm}[Corollary 3.1, \cite{GQS}]
Let $(X,\Z^k)$ be an extension of a free $Z^k$-action $(Y,\Z^k)$, $D\in\N\cup\{0\}$ and $L\in\N$. If ${\rm mdim}<L/2$, ${\rm dim}_{\rm Rok}(Y,\Z^k)=D$ and $Y$ has finite Lebesgue covering dimension, then there exists an embedding from $(X,\Z^k)$ into $(([0,1]^{(D+1)L+1})^{\Z^k},\sigma)$.
\end{thm}
Now by our Rokhlin dimension on more general groups, we can now try generalizing this result(though the dimension of the cube is still far from optimal).
\vspace{0.2cm}

\noindent {\bf Theorem E}\ (=Theorem \ref{ep})\  Let $(X, G)$ be a topological dynamical system, where $X$ is a compact metric space and $G$ a countably infinite group with $|Z(F)|=\infty$ for every finite subset $F\subset G$. Let $\{F_n\}\subset G$ be a  symmetric F$\o$lner sequence such that $(G, \{F_n\})$ satisfies (WAFC) with respect to $L_G$.

Let $m\in\N$, $d\in\N\cup\{0\}$ and $(Y, G)$ a factor of $X$. If $(Y, G)$ is free, ${\rm dim}(Y)\le d$ and ${\rm mdim}(X, G)<m/2$, then
\[{\rm edim}(X,G)\le(L_G(d+1)+1)m+1\stackrel{\triangle}{=}Q,\]
or in other words, there is an embedding from $(X, G)$ into $(([0,1]^Q)^G, \sigma)$. 

In particular, this applies to the cases $G=\mathbb{H}_{2n+1}(\Z)\times G_1$, where $\mathbb{H}_{2n+1}(\Z)$ is the discrete Heisenberg group, and $G_1$ is any of a finite group, a finitely generated abelian group, a locally finite group or a countable group satisfying (WAFC) with a symmetric F$\o$lner sequence.
\vspace{0.2cm}

Finally, we want to emphatically point out that the proof of Theorem A inherits a similar philosophy of Theorem 4.6 in \cite{Sza}. However, what is different is that, instead of taking $g_1, g_2,\cdots, g_d\in G$ so that $F^{-1}F, g_1F^{-1}F, \cdots, g_dF^{-1}F$ are pairwise disjoint in Lemma 4.3 of \cite{Sza}, we take turns to find a “sufficiently large” finite subset $\F$ containing $F$ and choose $g_i$ with the disjoint condition on both $\F^{-1}\F$ and $\F\F^{-1}$. This seemingly small change will actually give us a sufficiently nice property of the $F$-marker $O$ produced this way, so that it can also be $(\F,1)$-disjoint for any given finite set $\F\supset F$. Consequently, for any fixed group elements $h_1,h_2,\cdots,h_L$ and any finite sets $F\subset \F$, we can take an $F$-marker $O$ satisfying that every image $h_jO\,(j=1,2,\cdots,L)$ is also an $F$-marker, and even an $\F$-marker at the same time(although we don't know whether there exists an $F$-marker $O$ such that  image $gO$ is an $F$-marker for every $g\in G$, but this is not an issue for us). Finally, a nice marker makes a nice structure of a Rokhlin cover, with the help of approximate groups $\{F_n\}$.

\subsection{Outline of the paper}
The paper is organized as follows. 

Section \ref{sec:2} will provide definitions, including basic notions of amenable groups, topological dynamical systems of general group actions, marker property, $C^*$-level dimensions and some concepts of the embedding problem.

In Section \ref{sec:3}, the definition of the {\it (weakly) approximate groups} is given(Definition \ref{AG}), based on which we develop the notion of the {\it weakly approximate F$\o$lner condition(WAFC)}. Direct products definitely keep (WAFC), and under a certain assumption on the actions, (WAFC) of semidirect products can also be derived, as is shown in Theorem \ref{sdp}. This is particularly the case for Heisenberg groups $\mathbb{H}_{2n+1}(\Z)$.

Section \ref{sec:4} is devoted to the {\it topological Rokhlin dimension} and the {\it (controlled) marker property} of actions of groups with (WAFC) on compact metric spaces, generalizing G. Szab$\acute{\rm o}$'s Rokhlin dimension of $\Z^k$-actions. We prove that with a mild condition on the commutativity of $G$, every such free action on a finite dimensional space has the controlled marker property(Theorem \ref{cmk}), finite topological Rokhlin dimension(Corollary \ref{frd}), and even finite strong topological Rokhlin dimension(Corollary \ref{am}).

We consider several $C^*$-flavour dimensions of these dynamical systems in Section \ref{sec:5}. In particular, under the finiteness result of the strong topological Rokhlin dimension, we prove that for every such $G$ and $G$-actions on compact metric spaces $X$, the induced actions of $G$ on $C(X)$ have {\it finite amenability dimension}(Theorem \ref{fad}), {\it finite dynamic asymptotic dimension}, {\it finite tower and fine tower dimensions}. Consequently, the corresponding crossed products all have {\it finite nuclear dimensions}(Corollary \ref{nucdim}).

In Section \ref{sec:6}, similarly to Y. Gutman's method, we make an attempt on the embedding problem, applying our so-defined topological Rokhlin dimension. Then Theorem \ref{ep} is obtained this way. 

Finally, we discuss in Section \ref{sec:7} most of the {\it (topological) Rokhlin-type properties} that once appeared in the classification program of $C^*$-crossed products arising from dynamical systems and their implications. Several questions and conjectures are given. 

\section{Preliminaries}\label{sec:2}

Throughout the article, we will use $\N$ and $\N\cup\{0\}$ to denote the sets of natural numbers and nonnegative integers respectively. For a finite set $F$, we will write $|F|$ for its cardinality. 

The Letters $G, H, \cdots$ will usually be used to represent groups, $F, M,\cdots$ for finite subsets(in a group or in a $C^*$-algebra), $g, h,\cdots$ for group elements, $U, V, O\cdots$ for open subsets of a compact space, $x,y,\cdots$ for points in $X$, and $i,j,k,\cdots$ for indices.

For a $C^*$-algebra $A$, the notation $A_{1,+}$ is used to denote the set of positive elements in $A$ with norm at most $1$. For any two elements $a, b\in A$ and $\varepsilon>0$, we will write $a=_\varepsilon b$ if $\|a-b\|\le \varepsilon$. We also set $[a,b]=ab-ba$. 

For any $D\in\Z$, to make it convenient notationally, we will use $D^{+1}$ to stand for $D+1$.

\subsection{Amenable groups and F$\o$lner sequences}
Let $G$ be a group, $K\Subset G$ a finite subset and $\varepsilon>0$ a positive real number. We shall say that a finite subset $F\Subset G$ is {\it $(K,\varepsilon)$-invariant}, if $|KF\bigtriangleup F|\le \varepsilon|F|$. By a {\it F$\o$lner sequence in} $G$, we mean an increasing sequence $\{F_n\}$ of finite subsets of $G$ such that $e\in F_n$ for all $n\ge1$, $\bigcup_{n\ge1}F_n=G$ and 
\[\lim_{n\to\infty}\frac{|gF_n\bigtriangleup F_n|}{|F_n|}=0\]
for all $g\in G$. Given that $G$ is a countable group, it is said to be {\it amenable} if $G$ admits a F$\o$lner sequence. The groups considered in the present paper will all be countable.

\subsection{Growth rate of a finitely generated group}
Let $G$ be a finitely generated group with the set of generators $\mathcal{S}$. We may assume $\mathcal{S}$ is {\it symmetric}, in the sense that $x\in\mathcal{S}$ if and only if $x^{-1}\in\mathcal{S}$. Then any element $x\in G$ can be expressed as a {\it word} over the alphabet $\mathcal{S}$
\[x=a_1a_2\cdots a_k\ \ {\rm for}\ a_i\in\mathcal{S}.\]
Let $B_n(G,\mathcal{S})$ be the finite set of elements of $G$ that can be expressed by such a word of length no more than $n$. Define $\#(n)=|B_n(G,\mathcal{S})|$. 

If there is $C, k<\infty$ with $\#(n)\le C(n^k+1)$ for all $n\ge1$, 
then we say $G$ has a {\it polynomial growth}. If for some $C>1$, $\#(n)\ge C^n$ for all $n\ge1$, then $G$ is said to have an {\it exponential growth}. A group is said to be of {\it subexponential growth} if if it is {\it not} of exponential growth. Any group of subexponential growth is amenable. See \cite{Gr} for example.


\subsection{Continuous transformation groups}
Let $G$ be a countable group and $X$ be a topological space. We say that $(X, G)$ is a {\it (continuous) transformation group} or a {\it topological dynamical system}, if $G$ acts on $X$ by homeomorphisms, or in other words, there is a continuous map
\[F: G\times X\to X\]
such that $F(e,x)=x$ and $F(g_1g_2, x)=F(g_1, F(g_2,x))$ for all $x\in X$ and $g_1,g_2\in G$. For $g\in G$, we will always denote the map $x\mapsto F(g,x)$ by $x\mapsto gx$.

A dynamical system $(X,G)$ is said to be {\it free}, if no element in $G$ except the identity leaves some point where it is, and {\it minimal}, if every $G$-orbit is dense in $X$.

Let $(X, G)$ and $(Y, G)$ be two topological dynamical systems. We say that $(Y, G)$ is an {\it extension} of $(X, G)$, if there is a surjection map $\pi: Y\to X$, which we shall call a {\it factor map}, such that $\pi$ is {\it equivariant} in the sense that
\[\pi(gy)=g\pi(y)\]
for all $y\in Y$ and $g\in G$. In this case, $(X, G)$ is called a {\it factor} of $(Y, G)$.

\subsection{Towers, castles and disjointness}
Let $(X, G)$ be a topological dynamical system. By a {\it tower} in $X$, we mean a pair $(U,S)$ consisting of a (nonempty) subset $U$ of $X$, and a finite subset $S$ of $G$, such that the family 
\[\{gU: g\in S\}\]
is mutually disjoint. The set $U$ is called the {\it base} of the tower, and $S$ the {\it shape} of the tower. For $g\in S$, the sets $gU$ are called the {\it levels} of the tower. We shall say the tower $(U, S)$ is {\it open}, if $U$ is open, {\it closed} if $U$ is closed, and {\it clopen} if $U$ is clopen. A {\it castle} in $X$ is a finite collection of towers $(U_i, S_i)_{i\in I}$ that covers $X$, i.e.,
\[X=\bigcup_{i\in I}\bigcup_{g\in S_i}gU_i.\]
Note that for any two towers, we {\it don't} need them to be mutually disjoint.

Let $M\Subset G$ and $O\subset X$ be subsets, and $d\in\N$ be a natural number. We say that $O$ is {\it $(M,d)$-disjoint}, if whenever $g_1, g_2,\cdots, g_d\in M$, we have
\[g_1O\cap g_2O\cap\cdots\cap g_dO=\varnothing.\]
If $O$ is $(M,d)$-disjoint, then $O$ is obviously $(M',d)$-disjoint for any subset $M'\subset M$.

Given a tower $(U, S)$ in $X$, the base $U$ is then $(S,1)$-disjoint by definition. See also Definition 3.1 in \cite{Sza} for details.

\subsection{The orbit capacity}
Let $(X,G)$ be a topological dynamical system and $E\subset X$ a subset. We define the {\it orbit capacity} of $E$ in the following manner.
\[{\rm ocap}(E)=\inf_{F\Subset G}\sup_{x\in X}\frac{1}{|F|}\sum_{s\in F}1_{E}(sx).\]
If $G$ is a countable amenable group with F$\o$lner sequence $\{F_n\}$, then
\[{\rm ocap}(E)=\lim_{n\to\infty}\sup_{x\in X}\frac{1}{|F_n|}\sum_{s\in F_n}1_{E}(sx).\]
That the limit exists follows from Lemma 3.2 in \cite{KS}, where the orbit capacity is called the {\it upper Banach density}.

\subsection{The marker property in $\Z$-dynamical systems}
Let $X$ be a compact Hausdorff space and $T: X\to X$ a homeomorphism of $X$ onto itself. Let $O\subset X$ be an open set and $n\in\N$ a natural number. We say that $O$ is an {\it (open) $n$-marker}, if

$\bullet$ $O\cap T^i(O)=\varnothing$ for every $i\in\{0,1\cdots,n-1\}$;

$\bullet$ $\bigcup_{i\in\Z}T^i(O)=X$.

\noindent The system $(X, T)$ is said to have {\it the marker property}, if there exists an (open) $n$-marker for every $n\in\N$. See Definition 5.1 in \cite{Gut1} for reference. 

The marker property is stable under extensions, that is, if $(X,T)$ has the marker property, then so does any extension $(Y, T')$ of $(X,T)$.

\subsection{Amenability dimensions of $(X, G, \af)$ and $(C(X), G, \overline{\af})$}
Let $G$ be a countable group, $X$ a compact metric space and $(X, G)$ be a free dynamical system. Let $d\in\N\cup\{0\}$. 

Denote by $\Delta(G)$ the set of probability measures on $G$, and $\Delta_d(G)$ the set of probability measures on $G$ whose support has cardinality at most $d+1$.

The {\it amenability dimension} ${\rm dim}_{\rm am}(X,G)$(or ${\rm dim}_{\rm am}(\af)$) of the action $G\curvearrowright X$ is the smallest integer $d\ge0$ satisfying the property that, for every finite set $F\Subset G$ and $\varepsilon>0$, there is a continuous map $\varphi: X\to\Delta_d(G)$ such that for all $s\in F$,
\[\sup_{x\in X}\|\varphi(sx)-s\varphi(x)\|_1<\varepsilon.\]

Denote the induced action of $G$ on $C(X)$ by $\overline{\af}$. Then the {\it amenability dimension} of the action $G\curvearrowright C(X)$ is the smallest natural number $d$ with the following property.

For any $\varepsilon>0$, any finite subsets $M\Subset G$ and $F\Subset C(X)$, there exists finitely supported maps
\[\mu^{(l)}: G\to C(X)_{1,+}, g\mapsto \mu_g^{(l)}\ {\rm for}\ l=0,1,\cdots,d\]
satisfying that

(a) $(\sum_{0\le l\le d}\sum_{g\in G}\mu_g^{(l)})\cdot a=_\varepsilon a$ for any $a\in F$;

(b) $\mu_g^{(l)}\mu_h^{(l)}=0$ for all $l=0,1,\cdots,d$ and $g\ne h$ in $G$;

(c) $\|\sum_{h\in E}(\overline{\af}(\mu_h^{(l)})-\mu_{gh}^{(l)})\cdot a\|\le\varepsilon$ for all $a\in F$, $l=0,1,\cdots,d$, $g\in M$ and every $E\Subset G$;

(d) $\|[\sum_{g\in E}\mu_g^{(l)},a]\|\le\varepsilon$ for all $a\in F$, $l=0,1,\cdots,d$ and every $E\Subset G$.

\noindent If no such integer $d$ exists, then we set ${\rm dim}_{\rm am}(\overline{\af})=\infty$. it is known that in our setting, ${\rm dim}_{\rm am}(\af)\le d$ if and only if ${\rm dim}_{\rm am}(\overline{\af})\le d$. See Definition 5.1 in \cite{K}, Definition 6.1 and Lemma 8.4 in \cite{SWZ} for details and proofs.

\subsection{Dynamic asymptotic dimension of $(X,G)$}
Let $(X, G)$ be a topological dynamical system. The {\it dynamic asymptotic dimension ${\rm dad}(X,G)$} of $(X,G)$ is the smallest integer $d\ge0$ with the property that, for every finite set $E\Subset G$, there are a finite set $F\Subset G$ and an open cover $\mathcal{U}$ of $X$ of cardinality $d+1$ such that, for all $x\in X$ and $s_1,s_2,\cdots,s_n\in E$, if the points 
\[x, s_1x, s_2s_1x,\cdots,s_n\cdots s_1x\]
are contained in a common member of $\mathcal{U}$, then $s_n\cdots s_1\in F$. See Definition 5.3 in \cite{K}.

\subsection{Tower dimension and fine tower dimension}
Let $(X, G)$ be a topological dynamical system and $E\Subset G$ a finite subset. A collection of towers $\{(V_i,S_i)\}_{i\in I}$ covering $X$ is said to be {\it $E$-Lebesgue}, if for every $x\in X$, there are an index $i\in I$ and a $t\in S_i$ such that $x\in tV_i$ and $Et\subset S_i$.

Given a family $\mathcal{C}$ of subsets of a set, its {\it chromatic number} is defined as the least $d\in\N$ such that there is a partition of $\mathcal{C}$ into $d$ subcollections each of which is pairwise disjoint.

The {\it tower dimension} of $(X,G)$, denoted by ${\rm dim}_{\rm tow}(X,G)$, is the smallest number $d\in\N\cup\{0\}$ with the property that for every finite set $E\Subset G$, there is an $E$-Lebesgue collection of open towers $\{(V_i, S_i)\}_{i\in I}$ covering $X$ such that the family $\{S_iV_i\}_{i\in I}$ has chromatic number at most $d+1$. If no such $d$ exists, we set ${\rm dim}_{\rm tow}(X,G)=\infty$.

The {\it fine tower dimension} ${\rm dim}_{\rm ftow}(X,G)$ of $(X,G)$ is the least integer $d\ge0$ with the property that for every finite set $E\Subset G$ and $\delta>0$, there is an $E$-Lebesgue collection of open towers $\{(V_i, S_i)\}_{i\in I}$ covering $X$ such that ${\rm diam}(sV_i)<\delta$ for all $i\in I$ and $s\in S_i$, and the family $\{S_iV_i\}_{i\in I}$ has chromatic number at most $d+1$. If no such $d$ exists, we set ${\rm dim}_{\rm ftow}(X,G)=\infty$. See Definition 4.3, Definition 4.10 in \cite{K}.

\subsection{Nuclear dimension}
Let $A$ be a $C^*$-algebra. Recall that a linear map $\varphi: A\to B$ is {\it order zero}, if $\varphi(a_1)\varphi(a_2)=0$ for all $a_1,a_2\in A$ with $a_1a_2=0$.

The {\it nuclear dimension} of $A$, denoted by ${\rm dim}_{nuc}(A)$, is the smallest integer $d\ge0$ such that for every finite set $F\Subset A$ and $\varepsilon>0$, there are finite-dimensional $C^*$-algebras $B_0, B_1,\cdots, B_d$ and linear maps
\[A\stackrel{\varphi}{\longrightarrow}B_0\oplus B_1\oplus\cdots\oplus B_d\stackrel{\psi}{\longrightarrow}A\]
such that $\varphi$ is c.p.c., $\psi|_{B_i}$ is c.p.c. and order zero for each $i=0,1,\cdots,d$, and
\[\|\psi\circ\varphi(a)-a\|<\varepsilon\]
for all $a\in F$. If no such $d$ exists, we then write ${\rm dim}_{nuc}(A)=\infty$.

As we will only deal with crossed products, we note that the Toms-Winter conjecture holds for every simple crossed product arising from a free and minimal action of a countably infinite group $G$ on a compact metric space $X$ with small boundary property.  That is, for this class of dynamical systems, finite nuclear dimension implies the {\it $\mathcal{Z}$-stability} and the classifiability by the {\it Elliott invariant}, where $\mathcal{Z}$ is the Jiang-Su algebra. Readers who are interested may see Corollary 9.5 in \cite{KS} for reference.

\subsection{Comparison}
Let $(X,G)$ be a topological dynamical system. We say that $(X,G)$ has {\it comparison}, if for every nonempty open subsets $A,B\subset X$ satisfying $\mu(A)<\mu(B)$ for every $G$-invariant regular Borel probability measure $\mu$ on $X$, and every closed subset $C\subset A$, there exists a finite collection $\mathcal{U}$ of open sets of $X$ which covers $C$ and $s_U\in G$ for every $U\in\mathcal{U}$, such that the images $s_UU$ for $U\in\mathcal{U}$ are pairwise subsets of $B$. See Definition 3.2 in \cite{K}.

\subsection{The $m$-dimensional cubical shifts}
Let $K$ be a compact metric space and $G$ be a countable group. Let $X=K^G$, which is clearly a compact metric space as well. Define an action $\sigma: G\curvearrowright X$ by
\[(\sigma_g(x))_h=x_{hg}.\]
If $K=[0,1]^m$ is the $m$-dimensional cube for some $m\in\N$, then we say $(([0,1]^m)^G, \sigma)$ is an {\it $m$-dimensional cubical shift}. We will also say $([0,1]^m)^G$ is an $m$-dimensional cubical shift, if the action $\sigma$ is understood.

\subsection{Lebesgue covering dimension}
Let $X$ be a topological space and $\mathcal{U}$ is a finite open cover of $X$. Define the {\it order} ${\rm ord}(\mathcal{U})$ of $\af$ by
\[{\rm ord}(\mathcal{U})=-1+\max_{x\in X}\sum_{O\in\mathcal{U}}1_O(x).\]
Let 
\[D(\mathcal{U})=\min_{\mathcal{U}'\ {\rm refines}\ \mathcal{U}}{\rm ord}(\bt),\]
where by $\mathcal{U}'$ {\it refines} $\mathcal{U}$, we mean that $\mathcal{U}'$ is a finite open cover satisfying that every element of $\mathcal{U}'$ is contained in some element of $\mathcal{U}$. The {\it Lebesgue covering dimension} is given by
\[{\rm dim}(X)=\sup_{\mathcal{U}}D(\mathcal{U}),\]
where $\mathcal{U}$ runs over all finite open covers of $X$. It is well known that every compact metric space of dimension $d$ embeds into the {\it $2d+1$-dimensional cube} $[0,1]^{2d+1}$.

Let $\mathcal{U}$ and $\mathcal{V}$ be two open covers of $X$. The open cover $\mathcal{U}\vee\mathcal{V}$ is given by
\[\mathcal{U}\vee\mathcal{V}=\{U\cap V: U\in\mathcal{U}, V\in\mathcal{V}\}.\]
If $\varphi: X\to X$ is a homeomorphism, then $\varphi(\mathcal{U})=\{\varphi(U): U\in\mathcal{U}\}$.

\subsection{The embedding dimension}
Let $(X, G)$ and $(Y, G)$ be topological dynamical systems. If there exists an equivariant embedding of $X$ into $Y$, then we say $(X,G)$ {\it embeds into} $(Y, G)$, and denote $(X, G)\hookrightarrow (Y,G)$.

For the dynamical system $(X, G)$, its {\it embedding dimension} is a natural number ${\rm edim}(X, G)$ given by
\[{\rm edim}(X, G)=\min\{Q\in\N: (X, G)\hookrightarrow (([0,1]^Q)^G, \sigma)\}.\]
If $(X, G)$ does not embed into $([0,1]^Q)^G$ for any $Q\in\N$, we then write ${\rm edim}(X, G)=\infty$.

\subsection{The $\varepsilon$-embeddings and induced maps}
Let $(X,d)$ be a metric space and $Y$ a topological space. Let $f:X\to Y$ be a continuous map and $\varepsilon>0$ a positive real number. We shall say $f$ is an {\it $\varepsilon$-embedding}, if for every $y\in Y$, 
\[{\rm diam}(f^{-1}(y))<\varepsilon,\]
where the diameter of a subset $E\subset X$ is given by ${\rm diam}(E)=\sup_{x,y\in E}d_X(x,y)$.

For a topological dynamical system $(X,G)$ and continuous map $f: X\to[0,1]^m$, the induced map $I_f: X\to (([0,1])^m)^G$ is defined as
\[(I_f(x))_g=f(gx).\]
It is an equivariant map from $(X, G)$ to $((([0,1])^m)^G, \sigma)$.

\section{Discrete groups with F$\o$lner sequences of approximate groups}\label{sec:3}

\subsection{Approximate groups, (AFC) and (SAFC)}
\begin{df}\label{AG}
Let $G$ be a group, $A\Subset G$ a finite subset and $L>0$ a natural number. We call $A$ is a {\it (weakly) approximate (sub)group} in $G$ with respect to $L>0$, if 

(i) The group identity $e_G\in A$;

(ii) $A^{-1}A$ is covered by a union of $L$ left translates of $A$ by elements in $G$, i.e., there are group elements
\[g_1, g_2,\cdots, g_L\in G\] 
(depending on $A$ itself) with
\[A^{-1}A\subset \bigcup_{1\le i\le L} g_i\cdot A.\]
If, in addition, such elements $g_i\,(i=1,\cdots,L)$ could be chosen from $A$, then we say $A$ is a {\it strongly approximate (sub)group}.

In particular, since $L$ depends on $A$, we denote the smallest one of such positive constants by $L_A$. 
\end{df}

\begin{rem}
Approximate groups was originally defined to be a {\it symmetric} finite subset containing the identity which can be covered by $L$ left(or right) translates of itself, see Definition 3.8 in \cite{T} or Definition 1.2 in \cite{CHT} for reference. This is also why we call $A$ in Definition \ref{AG} a {\it weakly} approximate group, since we don't necessarily require the symmetry condition of $A$ or $\{g_1,g_2,\cdots,g_L\}$ for our use.
\end{rem}

\begin{prop}\label{const}
Let $A\Subset G$ be a finite subset such that $A$ and $A^{-1}$ are both approximate groups in $G$, with the constants $L_A$ and $L_{A^{-1}}$ respectively. Let $B=A^{-1}\cup A$. Then $B$ is a symmetric approximate group in $G$ with its constant
\[L_B\leq 2L_AL_{A^{-1}}+L_A+L_{A^{-1}}.\]
\end{prop}
\begin{proof}
It is clear that $B$ is symmetric in the sense that $B=B^{-1}$. Note that
\begin{align*}
B^2=&\,B^{-1}B=(A^{-1}\cup A)(A^{-1}\cup A)\\
=&\,(A^{-1}A^{-1})\cup (AA)\cup (A^{-1}A)\cup (AA^{-1}).
\end{align*}
Since the identity element $e_G\in A$, we have
\begin{align*}
AA&\subset AA^{-1}A\subset \bigcup_{1\le j\le L_{A^{-1}}} g_j' A^{-1}A\subset \bigcup_{1\le i\le L_A}\bigcup_{1\le j\le L_{A^{-1}}}g_{j}'g_i A,\\
A^{-1}A^{-1}&\subset A^{-1}AA^{-1}\subset \bigcup_{1\leq i\leq L_A} g_iAA^{-1}\subset \bigcup_{1\le i\le L_A}\bigcup_{1\le j\le L_{A^{-1}}} g_ig_j'A^{-1}.
\end{align*}
Also note that $A\subset B$ and $A^{-1}\subset B$. Therefore, 
\begin{align*}
B^2&\subset \left(\bigcup_{1\le i\le L_A}\bigcup_{1\le j\le L_{A^{-1}}}g_{j}'g_i A\right)\cup\left(\bigcup_{1\le i\le L_A}\bigcup_{1\le j\le L_{A^{-1}}}g_ig_j'A^{-1}\right)\cup\left(\bigcup_{1\le i\le L_A} g_iA\right)\cup\left(\bigcup_{1\le j\le L_{A^{-1}}} g_j'A^{-1}\right)\\
&\subset \left(\bigcup_{1\le i\le L_A}\bigcup_{1\le j\le L_{A^{-1}}}g_{j}'g_i B\right)\cup\left(\bigcup_{1\le i\le L_A}\bigcup_{1\le j\le L_{A^{-1}}}g_ig_j'B\right)\cup\left(\bigcup_{1\le i\le L_A} g_iB\right)\cup\left(\bigcup_{1\le j\le L_{A^{-1}}} g_j'B\right).
\end{align*}
This shows that $L_B\leq 2L_AL_{A^{-1}}+L_A+L_{A^{-1}}$.
\end{proof}

\begin{df}\label{AFC}
Let $G$ be a discrete countable group, $F_n\Subset G\,(n=1,2,\cdots)$ a sequence of finite subsets and $L_G$ a natural number. 

We say that $(G, \{F_n\})$ satisfies the {\it weakly approximate F$\o$lner condition(WAFC) with respect to $L_G$}, if $\{F_n\}$ is a  F$\o$lner sequence satisfying:

(i) For every $n\ge1$, the finite subset $F_n$ is an approximate group in $G$ with respect to $L_{F_n}$ in the sense of Definition \ref{AG};

(ii) The sequence of positive integers $\{L_{F_n}\}_{n\ge1}$ is upper bounded: $L_{F_n}\le L_G\,(n\ge1)$.

\vspace{0.2cm}

If $(G, \{F_n\})$ satisfies (WAFC) and in addition, for all $n,k\ge1$, 
\[F_{n+k}\subset F_n\cdot F_k,\]
then $(G, \{F_n\}$ is said to have {\it approximate F$\o$lner condition}, and is written as (AFC).

\vspace{0.2cm}
Finally, we shall say $(G, \{F_n\})$ has the {\it strongly approximate F$\o$lner condition}, and denoted by (SAFC) instead, if it satisfies (AFC) with the condition (i) being replaced by the following one:

(i)$'$ Every $F_n$ is a strongly approximate group in the sense of Definition \ref{AG}.
\end{df}

\begin{rem}
Although it seems that the definitions of (AFC) and (SAFC) are artificial, all the results in the paper are in fact based on the assumption of (WAFC). We need (AFC) and (SAFC) only for Theorem \ref{sdp}, i.e., to ensure that every such semidirect product has a F$\o$lner sequence of approximate groups.
\end{rem}

\begin{eq}\label{eg}
\quad\par
$\bullet$ Every finite group and every $\Z^m\,(m\ge1)$ satisfies (SAFC);

$\bullet$ Every finitely generated abelian group satisfies (SAFC) (also see Corollary \ref{dp});

$\bullet$ Every locally finite group $G=\bigcup_{n\ge1}G_n$ satisfies (WAFC) and condition (i)$'$;

$\bullet$ All the conditions of (WAFC), (AFC) and (SAFC) inherit under group isomorphisms.
\end{eq}

Recall that for a group $G$, we say that it is {\it nilpotent}, if there exists a series of normal subgroups
\[\{e_G\}=G_0\vartriangleleft G_1 \vartriangleleft\cdots\vartriangleleft G_n=G,\]
where $[G,G_{i+1}]\leq G_i$. 

If a group $G$ has a finite index subgroup $H\le G$ which is nilpotent, then we call $G$ {1tually nilpotent}.

Now suppose that $G$ is finitely generated with both $(G, \{F_n\})$ and $(G, \{F_n^{-1}\})$ satisfying (WAFC). By Proposition \ref{const}, $L_{F_n^{-1}\cup F_n}$ is bounded from above. From Corollary 11.2 of \cite{BGT}, there hence exists a constant $K'$ with respect to $K=\max\{L_{F_n^{-1}\cup F_n}: n\ge1\}$, a (sufficiently large) natural number $N\ge1$ and group element 
\[g_1^{(N)}, g_2^{(N)}, \cdots, g_K^{(N)}\in G\]
such that 
\[(F_N^{-1}\cup F_N)^2\subset \bigcup_{1\le i\le K} g_i^{(N)}(F_N^{-1}\cup F_N)\ \ {\rm and}\ \ B_{K'}(G, S)\subset F_N^{-1}\cup F_N,\]
where $S$ is the (symmetric) set of generators in $G$ and $B_{K'}(G, S)$ is the set of elements of $G$ with length no longer than $K'$. In particular, this follows that
\[|(F_N^{-1}\cup F_N)^2|\leq K\cdot|F_N^{-1}\cup F_N|\ \ {\rm and}\ \ B_{K'}(G, S)\subset F_N^{-1}\cup F_N,\]
and therefore $G$ is of polynomial growth. Consequently, this is equivalent to saying that $G$ is virtually nilpotent, see \cite{G}. To make a summary, we have the following proposition. 
\begin{prop}
Let $G$ be a discrete countable group with $(G, \{F_n\})$ satisfying (WAFC) for a symmetric F$\o$lner sequence $\{F_n\}$. If $G$ is finitely generated, then $G$ is virtually nilpotent. In particular, $G$ has polynomial growth, and of course, subexponential growth.
\end{prop}

\begin{rem}
Suppose that $G$ is not finitely generated. One may believe that, for every finitely generated subgroup $H<G$, $H$ is virtually nilpotent and of polynomial growth, and consequently every finitely generated subgroup of $G$ has subexponential growth.

However, let's take $H<G$ be a finitely generated subgroup. Then 
\[\{H_n=F_n\cap H: n\ge1\}\]
is an increasing sequence of symmetric finite subset of $H$ with $\bigcup_{n\ge1}H_n=H$. It is known from \cite{BGT} that, if there is a constant $K$ such that 
\[|H_n^2|\le|F_n^2\cap H|<K|F_n\cap H|=K|H_n|,\]
then $H$ is virtually nilpotent.  However, it is not straightforward that such a constant $K$ exists, even if for every $n\ge1$, $|F_n^2|<L_G|F_n|$, since it may happen that there exist elements $a,b\in F_n\setminus H$ with $ab\in H$ and
\[\limsup_{n\to\infty}\frac{|\{ab\in F_n^2\cap H: a,b\in F_n\setminus H\}|}{|F_n\cap H|}=\infty.\]
Therefore, we ask the following question.
\end{rem}
\begin{qn}\label{fgg}
Let $G$ be a discrete countably infinite group with $(G, \{F_n\})$ satisfying (WAFC) where $\{F_n\}\subset G$ is a symmetric F$\o$lner sequence. Does every finitely generated subgroup $H<G$ has subexponential growth?
\end{qn}

\subsection{Semidirect products of groups with (SAFC)}
In \cite{Ja}, D. Janzen investigates that under what kind of conditions so that there exists a “nice” shape of F$\o$lner sequence in the semidirect product of two amenable groups. In particular, enlightened by Theorem 3 in \cite{Ja}, we have the following Theorem on the problem of when the semidirect product of two amenable groups satisfies (AFC).
\begin{thm}\label{sdp}
Let $N$ and $H$ be countable groups. Let $\af: H\to {\rm Aut}(N)$ be an action of $H$ on $N$, and $G=N\rtimes_{\af} H$ be the (unimodular) semidirect product under $\af$. 

Suppose that $N$ and $H$ satisfy (SAFC) with $\{N_l\}\subset N$ and $\{H_l\}\subset H$ being the corresponding symmetric F$\o$lner sequences.
If there exists a positive integer $M\ge1$ and a function $\xi: \N\times\N\to \N$ such that the following properties hold:

(1) For all $(i,j)\in \N\times\N$, $\af_{H_i}(N_j)\subset N_{\xi(i,j)}$ and $\xi(i,i)\leq i\cdot M$;

(2) For every $(a,b)\in G$ and $\varepsilon>0$, there is an integer $K\ge1$ such that for all $l\ge K$,
\[|a\af_b(N_l)\cap N_l|>(1-\varepsilon)|N_l|;\]

(3) For every $l_1,l_2\ge1$, every $x\in N_{l_1+l_2}$ and $b\in H_{l_1}$,
\[N_{l_1}x\cap\af_{b}(N_{l_2})\neq\varnothing,\]
then $G$ satisfies (AFC). 
\end{thm}

\begin{proof}
For every $l\ge1$, define
\[F_l=N_l\times H_l.\]
Then $F_l$ are finite subsets of $G$. We now show that $\{F_l\}_{l\ge1}$ is a F$\o$lner sequence, that is, for any $(a,b)\in N\rtimes_\af H$, 
\[\lim_{l\to\infty}\frac{|(a,b)F_l\bigtriangleup F_l|}{|F_l|}=0.\]
Since
$|(a,b)F_l\bigtriangleup F_l|=|(a,b)F_l|+|F_l|-2|(a,b)F_l\cap F_l|=2(|F_l|-|(a,b)F_l\cap F_l|)$, it suffices to show that
\[\lim_{l\to\infty}\frac{|(a,b)F_l\cap F_l|}{|F_l|}=1.\]
Note that 
\begin{align*}
(a,b)(N_l\times H_l)\cap(N_l\times H_l)=&\,(a\af_b(N_l)\times bH_l)\cap(N_l\times H_l)\\
=&\,(a\af_b(N_l)\cap N_l)\times(bH_l\cap H_l),
\end{align*}
we then see that, for every $\varepsilon>0$, there is an integer $K\ge1$, such that
\[\frac{|(a,b)F_l\cap F_l|}{|F_l|}=\frac{|a\af_b(N_l)\cap N_l|\cdot|bH_l\cap H_l|}{|N_l|\cdot|H_l|}>(1-\varepsilon)\frac{|bH_l\cap H_l|}{|H_l|}\]
whenever $l\ge K$. Since $\{H_l\}$ is a F$\o$lner sequence, letting $l\to\infty$ yields that
\[\liminf_{l\to\infty}\frac{|(a,b)F_l\cap F_l|}{|F_l|}\ge 1-\varepsilon.\]
Since $\varepsilon>0$ is chosen arbitrarily and $|(a,b)F_l\cap F_l|\leq |F_l|$ always holds, we conclude
\[\lim_{l\to\infty}\frac{|(a,b)F_l\cap F_l|}{|F_l|}=1.\]
Consequently, $\{F_l\}_{l\ge1}$ is a F$\o$lner sequence.

Now take $l\ge1$. For any $(a,b)\in F_l=N_l\times H_l$, one has $(a,b)^{-1}=(\af_{b^{-1}}(a^{-1}), b^{-1})$, which follows that $F_l^{-1}\subset\af_{H_l^{-1}}(N_l^{-1})\times H_l^{-1}$, 
and further, since $N_l$ and $H_l$ are symmetric,
\begin{align*}
F_l^{-1}F_l&\subset\left(\af_{H_l^{-1}}(N_l^{-1})\times H_l^{-1}\right)\cdot(N_l\times H_l)\subset \left(\af_{H_l^{-1}}(N_l^{-1})\cdot\af_{H_l^{-1}}(N_l)\right)\times(H_l^{-1}H_l)\\
&=(\af_{H_l}(N_l))^2\times(H_l)^2\subset N_{\xi(l,l)}^2\times H_l^2\stackrel{\xi(l,l)\le lM}{\subset} N_{lM}^2\times H_l^2\stackrel{N_{l_1+l_2}\subset N_{l_1}\cdot N_{l_2}}{\subset} N_l^{2M}\times H_l^2\\
&\subset\left(\bigcup_{1\le i_1, i_2, \cdots, i_{2M-1}\leq L_N} \prod_{j=1}^{2M-1}n_{i_j}^{(l)}N_l\right)\times\left(\bigcup_{1\le i\le L_H}h_i^{(l)}H_l\right)\\
&=\bigcup_{1\le i_1, i_2, \cdots, i_{2M-1}\leq L_N}\bigcup_{1\le i\le L_H} \left(\prod_{j=1}^{2M-1}n_{i_j}^{(l)}, h_i^{(l)}\right)\cdot \left(\af_{(h_i^{(l)})^{-1}}(N_l)\times H_l\right)\\
&\stackrel{h_i^{(l)}\in H_l\ {\rm and}\ H_l^{-1}=H_l}{\subset} \bigcup_{1\le i_1, i_2, \cdots, i_{2M-1}\leq L_N}\bigcup_{1\le i\le L_H} \left(\prod_{j=1}^{2M-1}n_{i_j}^{(l)}, h_i^{(l)}\right)\cdot \left(N_{\xi(l,l)}\times H_l\right)\\
&\subset\bigcup_{1\le i_1, i_2, \cdots, i_{2M-1}\leq L_N}\bigcup_{1\le i\le L_H} \left(\prod_{j=1}^{2M-1}n_{i_j}^{(l)}, h_i^{(l)}\right)\cdot \left(N_l^M\times H_l\right)\\
&\subset\bigcup_{1\le i_1, i_2, \cdots, i_{2M-1}\leq L_N}\bigcup_{1\le i\le L_H}\bigcup_{1\leq k_1,k_2,\cdots,k_{M-1}\le L_N}\left(\prod_{j=1}^{2M-1}n_{i_j}^{(l)}, h_i^{(l)}\right)\cdot\left(\prod_{p=1}^{M-1}n_{k_p}^{(l)}, e_H\right)F_l,
\end{align*}
from which we see that 
\[L_{F_l}\leq L_N^{2M-1}\cdot L_H\cdot L_N^{M-1}=L_N^{3M-2}\cdot L_H\]
uniformly for $l\ge1$. This verifies conditions (ii) and (iii) in Definition \ref{AFC}. 

For (i) in Definition \ref{AFC}, let $l_1,l_2\ge1$. To show $F_{l_1+l_2}\subset F_{l_1}\cdot F_{l_2}$, we notice that
\[F_{l_1}\cdot F_{l_2}=\{(a_1\af_{b_1}(a_2), b_1b_2): (a_i,b_i)\in N_{l_i}\times H_{l_i}, i=1,2\}.\]
Take any $(x,y)\in F_{l_1+l_2}=N_{l_1+l_2}\times H_{l_1+l_2}$. Since $y\in H_{l_1+l_2}\subset H_{l_1}\cdot H_{l_2}$, there exist group elements $b_1\in H_{l_1}$ and  $b_2\in H_{l_2}$ with
\[y=b_1b_2.\]
By assumption (3), the finite sets $N_{l_1}x$ and $\af_{b_1}(N_{l_2})$ has nonempty intersection. That is, there exists $a_1\in N_{l_1}$ and $a_2\in N_{l_2}$ satisfying
\[a_1x=\af_{b_1}(a_2).\]
This yields that $x=a_1^{-1}\af_{b_1}(a_2)$, and hence by the symmetry of $N_{l_1}$,
\[(x,y)=(a_1^{-1}\af_{b_1}(a_2), b_1b_2)\in F_{l_1}\cdot F_{l_2}.\]
This verifies $F_{l_1+l_2}\subset F_{l_1}\cdot F_{l_2}$, and finally $G$ satisfies (AFC).
\end{proof}

\vspace{0.2cm}

\begin{eq}\label{Hei}
({\bf Discrete Heisenberg groups})\ Let $N=\Z^2$, $H=\Z$ and $\af: \Z\to {\rm Aut}(\Z^2)$ be defined by 
\[\af_l(a,b)=(a+lb,b)\]
for $(a,b)\in\Z^2$ and $l\in\Z$. Then it is known that $G=\Z^2\rtimes_\af\Z=\mathbb{H}_3(\Z)$ is the three-dimension discrete Heisenberg group.

Note that both $\Z^2$ and $\Z$ satisfy (SAFC), and by taking F$\o$lner sequences
\[N_l=([-\sqrt{l},\sqrt{l}]\times[-\sqrt{l},\sqrt{l}])\cap\Z^2\ \ {\rm and}\ \ H_l=[-\sqrt{l},\sqrt{l}]\cap\Z\]
will yield this. In fact, $L_{\Z^2}=4$ and $L_{\Z}=2$, and since $\sqrt{l_1+l_2}\leq \sqrt{l_1}+\sqrt{l_2}$, we have
\[H_{l_1+l_2}=[-\sqrt{l_1+l_2},\sqrt{l_1+l_2}]\cap\Z\subset([-\sqrt{l_1},\sqrt{l_1}]\cdot[-\sqrt{l_2},\sqrt{l_2}])\cap\Z=H_{l_1}\cdot H_{l_2}\]
and
\[N_{l_1+l_2}=([-\sqrt{l_1+l_2},\sqrt{l_1+l_2}])^2\cap\Z^2\subset [-\sqrt{l_1}-\sqrt{l_2}, \sqrt{l_1}+\sqrt{l_2}]^2\cap\Z^2=N_{l_1}\cdot N_{l_2}.\]
 Moreover, for every $(i,j)\in\N\times\N$, $h\in H_i=[-\sqrt{i},\sqrt{i}]$ and $n=(a,b)\in N_j=[-\sqrt{j},\sqrt{j}]^2$, one has
\[\af_{h}(a,b)=(a+bh,b)\in[-\sqrt{j}-\sqrt{ji},\sqrt{j}+\sqrt{ji}]\times[-\sqrt{j},\sqrt{j}]\subset N_{\sqrt{j}+\sqrt{ji}}.\]
Therefore, let $\xi_{(i,j)}=\sqrt{j}+\sqrt{ji}$ and it is clear that $\xi_{(i,i)}=\sqrt{i}+i\le 2i$.

To see (2), let $a=(m,k)\in\Z^2$ and $b\in\Z$. Then for every $l\in\N$, one immediately see
\[a\af_b(N_l)=[-(b+1)\sqrt{l}+m, (b+1)\sqrt{l}+m]\times[-\sqrt{l}+k, \sqrt{l}+k].\]
It is obvious that, for any given $\varepsilon>0$, we can choose sufficiently large $l$, so that
\[[-\sqrt{l},\sqrt{l}]\cap\Z\subset[-(b+1)\sqrt{l}+m, (b+1)\sqrt{l}+m]\cap\Z\]
and (assume $b>-1$ without loss of generality)
\[|[-\sqrt{l}+k,\sqrt{l}+k]\cap[-\sqrt{l},\sqrt{l}]\cap\Z|>(1-\varepsilon)|[-\sqrt{l},\sqrt{l}]\cap\Z|.\]
This will follow (2). 

Finally, let $x=(k,m)\in [-\sqrt{l_1+l_2},\sqrt{l_1+l_2}]^2\cap\Z^2$ and $b\in H_{l_1}=[-\sqrt{l_1},\sqrt{l_1}]\cap\Z$. To see (3), it is straightforward that (assume $b>-1$ without loss of generality)
\[N_{l_1}x=([-\sqrt{l_1}+k, \sqrt{l_1}+k]\times[-\sqrt{l_1}+m,\sqrt{l_1}+m])\cap\Z\]
and
\[\af_b(N_{l_2})=([-(b+1)\sqrt{l_2}, (b+1)\sqrt{l_2}]\times[-\sqrt{l_2},\sqrt{l_2}])\cap\Z.\]
Note that since $-\sqrt{l_1+l_2}\leq k\leq \sqrt{l_1+l_2}$ and $-\sqrt{l_1}\leq b\leq\sqrt{l_1}$, we see that none of
\begin{align*}
&\sqrt{l_1}+k<-(b+1)\sqrt{l_2}, \ \ (b+1)\sqrt{l_2}<-\sqrt{l_1}+k\\
&\sqrt{l_1}+m<-\sqrt{l_2},\ \ \sqrt{l_2}<-\sqrt{l_1}+m
\end{align*}
could happen. This follows that $N_{l_1}x\cap\af_b(N_{l_2})\neq\varnothing$. Therefore, $(\mathbb{H}_3(\Z), \{N_l\times H_l\})$ satisfies (AFC). Similarly, one can also check that $\{F_l^{-1}\}$ satisfies the condition of (WAFC) (but not necessarily (AFC)). Now we have the following corollary.
\end{eq}

\begin{cor}\label{hg}
For every $n\ge1$, there is an appropriate F$\o$lner sequence of approximate groups $\{F_l\}$ such that both $(\mathbb{H}_{2n+1}(\Z), \{F_l\})$ and $(\mathbb{H}_{2n+1}(\Z), \{F_l^{-1}\})$ satisfie (WAFC).
\end{cor}
\begin{proof}
The verification is similar to that in Example \ref{Hei} of the special case for $n=1$.
\end{proof}

\begin{cor}\label{dp}
Let $N$ and $H$ be countable groups with symmetric F$\o$lner sequences of strongly approximate groups. Then the direct product $G=N\times H$
satisfies (SAFC)((AFC) or (WAFC) respectively) if both of $N$ and $H$ do. 
\end{cor}
\begin{proof}
Assume $N$ and $H$ satisfies (SAFC). Since the direct product could be seen as the semidirect product under the trivial action $\af(H)=\{{\rm id}_N\}$, and it is clear that all the assumptions in Theorem \ref{sdp} hold for the trivial action, we see that $G=N\times H$ satisfies (AFC) by Theorem \ref{sdp}. 

On the other hand, that it satisfies (SAFC) follows from the straightforward verification
\[F_l^{-1}F_l\subset \bigcup_{1\le i\le L_N} \bigcup_{1\le j\le L_H}(n_i,h_j) N_l\times H_l=\bigcup_{1\le i\le L_N}\bigcup_{1\le j\le L_H}(n_i,h_j)F_l,\]
where $(n_i, j_j)\in N_l\times H_l=F_l$.

Finally, if $N$ and $H$ satisfy (AFC)(or (WAFC)), then for the trivial action $\af$ on $N$, every step in the proof of Theorem \ref{sdp} also works for $\af$. However, to avoid confusion here, we particularly specify that, in the proof of Theorem \ref{sdp}, the condition (i)$'$ in Definition \ref{AFC} is invoked only when we want to verify that $F_l$ is an approximate group, but this automatically holds for the trivial action).
\end{proof}

\section{Dynamical systems $(X, G)$ with $G$ satisfying (WAFC)}\label{sec:4}
\subsection{Topological Rokhlin dimension and (controlled) marker property}
\begin{df}\label{Rokdim}
Let $(X, G)$ be a topological dynamical system, where $X$ is a compact metric space, $G$ an amenable countable group acting on $X$.

Let $D\in\N\cup\{0\}$. We say that $(X,G)$ has {\it topological Rokhlin dimension} $D$, and denote ${\rm dim}_{\rm Rok}(X,G)=D$, if $D$ is the smallest natural number with the following property:

There is a F$\o$lner sequence $\{F_n\}\subset G$ such that for every $n\in\N$, there exist open subsets $U_0,U_1,\cdots,U_D\subset X$ satisfying

(i) For every $0\le i\le D$, $\{g\overline{U_i}: g\in F_n\}$ is a disjoint (closed) tower with the shape $F_n$;

(ii) The union of these (open) towers covers $X$, that is, $\bigcup_{0\le i\le D} F_nU_i=X$.
\end{df}

\begin{df}[Definition 4.1, \cite{Sza}]\label{mk}
Let $(X, G)$ be a topological dynamical system, $F\Subset G$ a finite subset and $O\subset X$ an open set. We say $O$ is an {\it $F$-marker}, if

(i) The collection $\{g\overline{O}: g\in F\}$ is mutually disjoint;

(ii) $X=\bigcup_{g\in G}gO$.

Moreover, for a countable amenable group $G$ with $\{F_n\}$ being its F$\o$lner sequence,  we shall say $(X,G)$ has the {\it marker property}, if for every $n\ge1$, $X$ admits an $F_n$-marker.
\end{df}

\begin{df}
Let $(X, G)$ be a topological dynamical system, $F\Subset G$ a finite subset and $L\in\N$. Let $O\subset X$ be an open set. We say $O$ is a {\it right $L$-controlled $F$-marker}, if $O$ satisfies the following conditions:

(i) The open set $O$ is an $F$-marker in the sense of Definition \ref{mk};

(ii) There exists a finite subset $B=\{g_1, g_2,\cdots,g_L\}\subset G$ such that
\[X=\bigcup_{i=1}^L \bigcup_{g\in F}gg_iO=(FB)O,\]
that is, $X$ is covered by the image of $O$ under a union of $L$ right translates of $F$. 

Moreover, for a countable amenable group $G$ with $\{F_n\}$ being its F$\o$lner sequence,, we shall say $(X,G)$ has the {\it $L$-controlled marker property}, if for every $n\ge1$, $X$ admits a right $L$-controlled $F_n$-marker.
\end{df}

\begin{rem}
Intuitively speaking, by the existence of an $L$-controlled $F$-marker $O$, we mean that $X$ satisfies the following: 

(i) The $F^{-1}$-orbit $\{gx: g\in F^{-1}\}$ of every point $x\in X$ intersects $\overline{O}$ at most once: $g_1\overline{O}\cap g_2\overline{O}\ne\varnothing$ if and only if there exists $x_1, x_2\in \overline{O}$ with $y=g_1x_1=g_2x_2$, if and only if $g_1^{-1}y=x_1\in \overline{O}$ and $g_2^{-1}y=x_2\in \overline{O}$;

(ii) Every orbit in $X$ intersects $O$ within a union of $L$ left translates $B^{-1}F^{-1}$ of $F^{-1}$.
\end{rem}

\begin{df}[Definition 3.2, \cite{Sza}]
Let $G$ be a countably infinite group and $d\in\N\cup\{0\}$ be a natural number. 

A topological dynamical system $(X, G)$ is said to has the {\it bounded topological small boundary property with respect to $d$}, abbreviated $(TSBP\le d)$, if whenever $U, V\in X$ are open sets with $\overline{U}\subset V$, we can find $\overline{U}\subset U_0\subset V$ open such that $\partial U_0$ satisfies the following:

The nonnegative number $d$ is the smallest one such that, for every subset of $d$ distinct elements $\{\gamma_0,\gamma_1,\cdots,\gamma_d\}\subset G$, we have
\[\bigcap_{i=0}^d\gamma_i\partial U_0=\varnothing.\]
\end{df}
\begin{rem}
This is rather a strong condition compared to the usual definition of {\it small boundary property}. Recall that a topological dynamical system $(X, G)$ has the {\it small boundary property} if there is a basis for the topology on $X$ consisting of open sets $U$ such that $\overline{D}(\partial U)=0$, where
\[\overline{D}(\partial U)=\inf_{F\Subset G}\sup_{x\in X}\frac{1}{|F|}\sum_{s\in F}1_{\partial U}(sx).\]
The condition $(TSBP\le d)$ follows that for every finite subset $F$ and every $x\in X$, there are at most $d$ distinct points $\{g_ix: g_i\in F, i=0,1,\cdots,d-1\}\subset\partial U$, and therefore
\[\overline{D}(\partial U)\le\inf_{F\Subset G}\sup_{x\in X}\frac{d}{|F|}=0.\]
In fact, there is a version of small boundary property called {\it topological small boundary property}, which is slightly stronger than $(TSBP\le d)$, but also weaker than the usual SBP. See Definition 7.1 in \cite{KS} or Definition 3.2 in \cite{Sza} for reference.
\end{rem}

The following lemma will seem to be a little bit different from Lemma 4.3 in \cite{Sza}(but of great importance for the construction of noncommutative Rokhlin covers). Therefore, to avoid confusion and make it clear, we decide to give the adapted proof here. Readers who are interested may compare it to Lemma 4.3 in \cite{Sza}. In fact, what we need is a “sufficiently larger symmetric” $M$, so that $g_i$ can be put to the right. This will allow us to remedy the downside of non-commutativity of $G$ when we define the Rokhlin towers.

We recall that for subsets $E\subset X$, $M\subset G$ and a number $k\subset \N\cup\{0\}$, we say $E$ is {\it $(M, k)$-disjoint}, if for any $k+1$ distinct group elements $g_0, g_1,\cdots,g_k\in M$, 
\[\bigcap_{0\le i\le k} g_iE=\varnothing.\]

\begin{lem}\label{4.3}
Let $X$ be a compact metric space, $G$ a countably infinite group and $d\in\N$ a natural number. Let $(X, G)$ be a topological dynamical system satisfying $(TSBP\le d)$. 

Suppose that $e\in F\Subset G$ is a finite subset and $e=h_0, h_1, \cdots, h_L$ are chosen arbitrarily. Define
\[\tilde{F}=\bigcup_{0\le j\le L}Fh_j.\]
Take $e=g_0,g_1,\cdots,g_d\in G$ with the property that the following two families (note that this is feasible since $G$ is infinite)
\[\{\tilde{F}^{-1}\tilde{F},\ \tilde{F}^{-1}\tilde{F}g_1,\ \cdots,\ \tilde{F}^{-1}\tilde{F}g_d\}\ \ {\rm and}\ \ \{\tilde{F}^{-1}\tilde{F},\ g_1\tilde{F}^{-1}\tilde{F},\ \cdots,\ g_d\tilde{F}^{-1}\tilde{F}\}\]
are both mutually disjoint. Set 
\[M=\bigcup_{0\le i\le d}(\F^{-1}\F g_i\cup g_i\F^{-1}\F\cup g_i^{-1}\F^{-1}\F\cup \F^{-1}\F g_i^{-1})\ {\rm and}\ M_1=\bigsqcup_{0\le i\le d} FF^{-1} g_i.\]

Note that $M=M^{-1}$. Let $U,V\subset X$ be open sets such that the families $\{g\overline{U}: g\in \F\}$ and $\{g\overline{V}: g\in M\}$ are both pairwise disjoint. Then there exists an open set $W\subset X$ with

$\bullet$ $U\subset W$ and $V\subset \bigcup_{g\in M_1}gW$;

$\bullet$ $\{g\overline{W}: g\in \F\}$ is pairwise disjoint.
\end{lem}
\begin{proof}
Since $\overline{U}$ is $(\F, 1)$-disjoint, we can take a positive number $\varepsilon>0$ such that $\overline{B}_\varepsilon(U)=\{y\in X: d(y, U)\le \varepsilon\}$ is $(\F, 1)$-disjoint. The assumption of $(TSBP\le d)$ yields a new open set $U\subset U'\subset B_\varepsilon(U)$ such that $\partial U'$ is $(G, d)$-disjoint. We will still use $U$ to denote this new open set, by abuse of notation. Therefore, without loss of generality, we may assume $U$ be such that $\partial U$ is $(G, d)$-disjoint.

Let $R=\overline{V}\setminus \bigcup_{g\in M_1}gU$. Since $R\subset \overline{V}$, we know that $R$ is $(M,1)$-disjoint. Therefore, take $\rho>0$ so that $\overline{B}_\rho(R)$ is also $(M,1)$-disjoint. Now we see that there is a $\delta>0$ such that for all $x\in R$, 
\[|\{g\in M: g\overline{U}\cap\overline{B}_\delta(x)\ne\varnothing\}|<d.\]
In fact, if this doesn't hold, we can choose a sequence $\delta_n>0$ with $\delta_n\to 0$, and a sequence $x_n\in R$ converging to some $x\in R$ such that for all $n\ge1$, 
\[|\{g\in M: g\overline{U}\cap\overline{B}_{\delta_n}(x_n)\ne\varnothing\}|\ge d.\]
Then we  assume, without loss of generality that, there are distinct elements $\gamma_0,\gamma_1,\cdots, \gamma_d\in M$ such that $\gamma_i(\overline{U})\cap\overline{B}_{\delta_n}(x_n)\ne\varnothing$ for all $n\ge1$ and $l=0,1,\cdots,d$. Since $\delta_n\to0$, we see 
\[x\in R\cap\bigcap_{0\le l\le d}\gamma_l\overline{U}\subset \bigcap_{0\le l\le d}\gamma_l\partial U.\]
However, since $\partial U$ is $(G, d)$-disjoint, the right most one has to be an empty set, which comes to a contradiction. Thus, there is a $\delta<\rho$ with the above property. 

Now fix a finite open covering $R\subset \bigcup_{1\le i\le s}B_\delta(z_i)$ for some $z_1,z_2,\cdots,z_s\in R$. Note that the right side is $(M,1)$-disjoint, since $\delta<\rho$ and $\overline{B}_\rho(R)$ is $(M,1)$-disjoint. Now also note that, since the subfamily $\{g_i\F^{-1}\F: 0\le i\le d\}\subset M$ is pairwise disjoint, there is a map $c:\{1,\cdots,s\}\to\{0,\cdots,d\}$ such that for all $g\in g_{c(i)}\F^{-1}\F$,
\[g\overline{U}\cap\overline{B}_\delta(z_i)=\varnothing.\]
Then we let 
\[W=U\cup\bigcup_{1\le i\le s}g_{c(i)}^{-1}(B_\delta(z_i)).\]
It is clear that $U\subset W$. To show $V\subset\bigcup_{g\in M_1}gW$, note that
\begin{align*}
V&\subset R\cup \bigcup_{g\in M_1}gU\subset \bigcup_{1\le i\le s}B_\delta(z_i)\cup \bigcup_{g\in M_1}gU\\
&\subset \bigcup_{1\le i\le s} g_{c(i)}W\cup\bigcup_{g\in M_1}gU\subset \bigcup_{g\in M_1}gW\ ({\rm since}\ U\subset W).
\end{align*}
Finally, we show that $\overline{W}$ is $(\F,1)$-disjoint. If there are distinct $a,b\in \F$ such that $a\overline{W}\cap b\overline{W}\ne\varnothing$,  then we take $x,y\in\overline{W}$ with $a(x)=b(y)$.  This is divided into the following three situations:

(i) If $x,y\in\overline{U}$, this follows that $\overline{U}$ is not $(\F,1)$-disjoint, a contradiction;

(ii) If $x\in g_{c(i_1)}^{-1}(\overline{B}_\delta(z_{i_1}))$ and $y\in g_{c(i_2)}^{-1}(\overline{B}_\delta(z_{i_2}))$, then 
\[a(x)=b(y)\in ag_{c(i_1)}^{-1}(\overline{B}_\delta(z_{i_1}))\cap bg_{c(i_2)}^{-1}(\overline{B}_\delta(z_{i_2}))\ne\varnothing.\]
This follows that 
\[\varnothing\ne b^{-1}ag_{c(i_1)}^{-1}(\overline{B}_\delta(z_{i_1}))\cap g_{c(i_2)}^{-1}(\overline{B}_\delta(z_{i_2}))\subset b^{-1}ag_{c(i_1)}^{-1}(\overline{B}_\rho(R))\cap g_{c(i_2)}^{-1}(\overline{B}_\rho(R)).\]
Since $\overline{B}_\rho(R)$ is $(M,1)$-disjoint and $b^{-1}ag_{c(i_1)}^{-1}, g_{c(i_2)}^{-1}\in M$(this is exactly how we benefit from the “extension” of $M_1$ to $M$), we have $b^{-1}ag_{c(i_1)}^{-1}=g_{c(i_2)}^{-1}$ and hence 
\[g_{c(i_2)}b^{-1}a=g_{c(i_1)}.\]
By the assumption, $\{g_i\F^{-1}\F: i=0,1,\cdots,d\}$ is also disjoint, this will follow that $c(i_1)=c(i_2)$ and hence $b=a$, a contradiction;

(iii) If $x\in\overline{U}$ and $y\in g_{c(i)}^{-1}(\overline{B}_\delta(z_i))$ for some $i$, then 
\[a(x)=b(y)\in a(\overline{U})\cap bg_{c(i)}^{-1}(\overline{B}_\delta(z_i))\ne\varnothing,\]
which implies that $g_{c(i)}b^{-1}a(\overline{U})\cap \overline{B}_\delta(z_i)\ne\varnothing$. But this contradicts to the choice of $c(i)$. Therefore, the family $\{g\overline{W}: g\in \F\}$ is disjoint, which completes the proof.
\end{proof}

\begin{rem}\label{rm1}
For any finite subsets $e\in F\subset \F\Subset G$, one can always choose group elements $h_1, h_2,\cdots,h_L\in G$ such that 
\[\F\subset\bigcup_{1\le j\le L} h_jF.\]
In fact, simply setting $\{h_j: j=1,2,\cdots,L\}=\F$ will work. Therefore, Lemma \ref{4.3} makes sense for any finite subset $\F$ containing $F$.
\end{rem}

\begin{thm}\label{4.4}
Let $G$ be a countable group, $X$ a compact metric space, and $d\in \N$ a natural number. 

Let $(X, G)$ be a free topological dynamical system satisfying $(TSBP\le d)$. Let 
\[F, e=h_0, h_1,\cdots, h_L, e=g_0, g_1,\cdots, g_d, \F, M, M_1\]
be as in Lemma \ref{4.3}. Then there exists an open set $O\subset X$ such that

$\bullet$ $\overline{O}$ is $(\F, 1)$-disjoint;

$\bullet$ $X=\bigcup_{g\in M_1}gO$.
\end{thm}

\begin{proof}
Since $(X, G)$ is free, for every $x\in X$, we can choose an open subset $U_x\ni x$ such that the family
\[\{g\overline{U_x}: g\in M^{-1}=M\}\]
is pairwise disjoint. The compactness of $X$ yields a finite subcovering $X=\bigcup_{0\le i\le s}U_i$. 

Set $U=U_0$ and $V=U_1$. From Lemma \ref{4.3}, choose an open set $W_1\subset X$ such that $U_0\subset W_1$, $U_1\subset \bigcup_{g\in M_1}gW_1$ and $\overline{W_1}$ is $(\F,1)$-disjoint. Then we have
\[U_0\cup U_1\subset W_1\cup\bigcup_{g\in M_1}gW_1\subset\bigcup_{g\in M_1}gW_1,\]
since $e\in M_1$. Now we repeat this procedure. Suppose that $W_k$ is already defined. Set $U=W_k$ and $V=U_{k+1}$. Lemma \ref{4.3} again gives us $W_{k+1}$ such that $U=W_k\subset W_{k+1}$, $V=U_{k+1}\subset \bigcup_{g\in M_1}gW_{k+1}$ and $\overline{W_{k+1}}$ is $(\F,1)$-disjoint. Note that since $W_k\subset W_{k+1}$, 
\begin{align*}
U_0\cup\cdots\cup U_k\cup U_{k+1}&\subset U_{k+1}\cup\bigcup_{g\in M_1}gW_k\\
&\subset \bigcup_{g\in M_1}gW_{k+1}\cup\bigcup_{g\in M_1}gW_k\subset\bigcup_{g\in M_1}gW_{k+1},
\end{align*}
Set
\[O=W_s.\]
Then $\overline{O}=\overline{W_s}$ is $(\F,1)$-disjoint and 
\[X=U_0\cup\cdots\cup U_s\subset\bigcup_{g\in M_1}gW_s=\bigcup_{g\in M_1}gO,\]
which completes the proof.
\end{proof}

\begin{cor}\label{4.5}
Let $G$ be a countable group, $X$ a compact metric space and $d\in \N$ a natural number. Suppose that $(X, G)$ is a free topological dynamical system satisfying $(TSBP\le d)$.

Let $F\Subset G$ be any (nonempty) finite subset, and $e=h_0, h_1,\cdots, h_L$ be arbitrary group elements. Let $e=g_0, g_1,\cdots, g_d, M_1, M$ be as in Lemma \ref{4.3}. Then there is an open set $O$ with the following properties:

(1) $O$ is an $Fh_j$-marker for all $j=1,2,\cdots,L$;

(2) $X=\bigcup_{g\in M_1}gO$.
\end{cor}

\begin{proof}
By Theorem \ref{4.4}, there is an open set $O\subset X$ such that 

$\bullet$ $\overline{O}$ is $(\F, 1)$-disjoint;

$\bullet$ $X=\bigcup_{g\in M_1}gO$.

\noindent Therefore, it suffices to show that $O$ is an $Fh_j$-marker for all $j=1,2,\cdots,L$, or in other words, $\overline{O}$ is $(Fh_j,1)$-disjoint. But this is immediate, since $\tilde{F}=\bigcup_{0\le k\le L}Fh_k\supset Fh_j$.
\end{proof}

\begin{rem}
By Corollary \ref{4.5}, what we want to prove is that, we have not only the existence of an $F$-marker $O$ for a given finite set $F$, but also one for any given finite collection of right translates of $F$.
\end{rem}

\begin{rem}
As what we mentioned in Remark \ref{rm1}, almost without any adaption, the same procedure applies to Theorem \ref{4.4} and Corollary \ref{4.5} for any finite subset $\F$ that contains $F$, therefore our results also hold for any such $\F$. 
\end{rem}

\subsection{Finite Rokhlin dimension of $(X, G)$}
Let $G$ be a group and $J\subset G$ be a subset(not necessarily finite). Define
\[Z(J)=\{g\in G: gz=zg\ {\rm for\ all\ }z\in J\}.\]
Note that if $J=G$, then $Z(G)$ is the usual concept of the center of $G$. It is clear that for any $J\subset G$, $Z(J)$ is always a subgroup of $G$.

\begin{lem}\label{ct}
Let $G$ be an infinite group, $F\Subset G$ a finite subset and $d\in\N$ a natural number. 

If $H\subset G$ is an infinite subset, then we can always choose $g_1, \cdots, g_d\in H$ such that the family
\[\{F, Fg_1, \cdots, Fg_d\}\ {\rm and}\ \{F, g_1F, \cdots, g_dF\}\]
are pairwise disjoint.
\end{lem}
\begin{proof}
For $i=1$, we choose $g_1\in H$ such that
\[g_1\in (F^{-1}F)^c\cap(FF^{-1})^c.\]
Since $H$ is an infinite set and $F$ is finite, this is possible. Now assume that $g_1, g_2, \cdots, g_k$ has been chosen so that
\[\{Fg_0, Fg_1, \cdots, Fg_k\}\ {\rm and}\ \{g_0F, g_1F, \cdots, g_kF\}\ ({\rm with\ }g_0=e)\]
are pairwise disjoint. We then choose $g_{k+1}\in H$ such that
\[g_{k+1}\in \bigcap_{0\le i\le k}(F^{-1}Fg_i)^c\cap(g_iFF^{-1})^c.\]
Note that this is also feasible since $F^{-1}Fg_i$ and  $g_iFF^{-1}$ are finite for every $0\le i\le k$.
\end{proof}

\begin{thm}\label{cmk}
Let $(X, G)$ be a topological dynamical system, where $X$ is a compact metric space and $G$ a countably infinite group with $|Z(F)|=\infty$ for every finite subset $F\subset G$. Let $\{F_n\}\subset G$ be a  symmetric F$\o$lner sequence such that $(G, \{F_n\})$ satisfies (WAFC) with respect to $L_G$.

Suppose that $(TSBP\le d)$. Then for any $n\ge1$, there exists group elements 
\[v_1, v_2, \cdots, v_L\in G\ {\rm and}\ g_0, g_1, \cdots, g_d\in G\]
and an open set $O$ such that

$\bullet$ For every $1\le i\le L$ and $0\le j\le d$, $O$ is an $F_nv_i^{-1}g_j$-marker;

$\bullet$ $X=\bigcup_{0\le j\le d}\bigcup_{1\le i\le L}\bigcup_{g\in F_n}gv_i^{-1}g_jO$.

\noindent In particular, $G$ has the $L_G(d+1)$-controlled marker property, where $L_G$ is the constant in Definition \ref{AFC}.
\end{thm}
\begin{proof}
Fix $n\ge1$. Since $\{F_n^{-1}\}$ are approximate groups with a uniform constant $L=L_G$, there are elements $v_1, v_2,\cdots, v_L\in G$ for which
\[F_nF_n^{-1}\subset \bigcup_{1\le i\le L}v_iF_n^{-1}.\]
Taking the inverse yields that
\[F_nF_n^{-1}\subset\bigcup_{1\le i\le L}F_nv_i^{-1}.\]
Let $\tilde{F_n}=\bigcup_{0\le i\le L}F_nv_i^{-1}\subset G$ with $v_0=e$. Note that $\tilde{F_n}$ is a finite set. By the assumption, we know that $|Z(\tilde{F_n})|=\infty$. By Lemma \ref{ct}, we can then choose elements $e=g_0, g_1,\cdots, g_d\in Z(\tilde{F_n})$ such that the families
\[\{\tilde{F_n}^{-1}\tilde{F_n}g_0, \tilde{F_n}^{-1}\tilde{F_n}g_1, \cdots, \tilde{F_n}^{-1}\tilde{F_n}g_d\}\ {\rm and}\ \{g_0\tilde{F_n}^{-1}\tilde{F_n}, g_1\tilde{F_n}^{-1}\tilde{F_n}g_1, \cdots, g_d\tilde{F_n}^{-1}\tilde{F_n}\}\]
are both pairwise disjoint. Let 
\[M_1=\bigcup_{0\le j\le d}F_nF_n^{-1}g_j.\]
By Corollary \ref{4.5}, there is an open set $O$ with the following properties:

$\bullet$ $O$ is an $F_nv_i^{-1}$-marker for all $i=1,2,\cdots, L$;

$\bullet$ $X=\bigcup_{g\in M_1}gO$.

\noindent Note that we have
\begin{align*}
M_1=\bigcup_{0\le j\le d}F_nF_n^{-1}g_j&\subset\bigcup_{0\le j\le d}\bigcup_{1\le i\le L}F_nv_i^{-1}g_j,
\end{align*}
which shows that
\[X=\bigcup_{g\in M_1}gO\subset\bigcup_{0\le j\le d}F_nF_n^{-1}g_jO\subset\bigcup_{0\le j\le d}\bigcup_{1\le i\le L}\bigcup_{g\in F_n}gv_i^{-1}g_jO.\]
To show that $O$ is an $F_nv_i^{-1}g_j$-marker, we note that since $g_j\in Z(\tilde{F_n})\subset Z(F_nv_i^{-1})$, we have $gv_i^{-1}g_j=g_jgv_i^{-1}$ for all $g\in F_n, i=1,2,\cdots, L$ and $j=0,1,\cdots, d$. Suppose that
\[g_jgv_i^{-1}(\overline{O})\cap g_jg'v_i^{-1}(\overline{O})\ne\varnothing\]
for some distinct $g,g'\in F_n$, then clearly $gv_i^{-1}(\overline{O})\cap g'v_i^{-1}(\overline{O})\ne\varnothing$, which contradicts to the assumption that $\overline{O}$ is $(F_nv_i^{-1}, 1)$-disjoint.

Finally, to show that $G$ has the $L_G(d+1)$-controlled marker property, we set $B=\{v_i^{-1}g_j: 0\le j\le d, 1\le i\le L_G\}$(note that $L=L_G$), which completes the proof.
\end{proof}

\begin{cor}\label{frd}
Let $(X, G)$ be a topological dynamical system, where $X$ is a compact metric space and $G$ a countably infinite group with $|Z(F)|=\infty$ for every finite subset $F\subset G$. Let $\{F_n\}\subset G$ be a  symmetric F$\o$lner sequence such that $(G, \{F_n\})$ satisfies (WAFC) with respect to $L_G$.

If $(X,G)$ is free and ${\rm dim}(X)<\infty$, then $(X,G)$ has finite topological Rokhlin dimension. In particular,
\[{\rm dim}_{\rm Rok}(X,G)\le L_G({\rm dim}(X)+1)-1.\]
\end{cor}
\begin{proof}
Let $n\in\N$. Since $(X, G)$ is free, applying Theorem 3.8 in \cite{Sza}, there is a natural number $d={\rm dim}(X)$ for which $(X,G)$ satisfies $(TSBP\le d)$. By Theorem \ref{cmk}, there are group elements
\[v_1, v_2,\cdots, v_L\in G\ \ {\rm and}\ \ g_0, g_1,\cdots, g_d\in G\]
and an open set $O\subset X$ such that

$\bullet$ The natural number $L=L_G$;

$\bullet$ For every $1\le i\le L$ and $0\le j\le d$, $O$ is an $F_nv_i^{-1}g_j$-marker;

$\bullet$ $X=\bigcup_{0\le j\le d}\bigcup_{1\le i\le L}\bigcup_{g\in F_n}gv_i^{-1}g_jO$.

\noindent Now for every $0\le j\le d$ and $1\le i\le L$, we define the base $U_{i,j}$ and shape $S_{i,j}$ by
\[U_{i,j}=v_i^{-1}g_jO\ \ {\rm and}\ \ S_{i,j}=F_n.\]
 Then we have
 \[X=\bigcup_{0\le j\le d}\bigcup_{1\le i\le L}F_nU_{i,j}.\]
 Now it suffices to show that each tower in the so defined castle is indeed pairwise disjoint. But note that $O$ is an $F_nv_i^{-1}g_j$-marker for every $0\le j\le d$ and $1\le i\le L$, which follows by definition that $\overline{O}$ is $(F_nv_i^{-1}g_j, 1)$-disjoint, or in other words, $v_i^{-1}g_j\overline{O}=\overline{v_i^{-1}g_jO}$ is $(F_n,1)$-disjoint.
 
 Finally, since there are $L(d+1)$ towers, where $L$ and $d$ are numbers independent to $F_n$, we have ${\rm dim}_{Rok}(X,G)\le L(d+1)-1=L_G({\rm dim}(X)+1)-1<\infty$.
\end{proof}
Now by combining Corollary \ref{hg} and Corollary \ref{frd}, we have the following Corollary.

\begin{cor}\label{hsb}
Let $n\ge1$ and $X$ be a compact metric space with finite covering dimension. Then every free topological dynamical system $(X, \mathbb{H}_{2n+1}(\Z)\times G)$ has finite topological Rokhlin dimension, where $\mathbb{H}_{2n+1}(\Z)$ is the discrete Heisenberg group, and $G$ is any of a finite group, a finitely generated abelian group, a locally finite group or a countable group satisfying (WAFC) with a symmetric F$\o$lner sequence.
\end{cor}

\begin{cor}\label{am}
In Theorem \ref{cmk}, if we choose sufficiently large $N$ and  $N'$ such that,
\[F_n\subset \F_n=\bigcup_{0\le i\le L}F_nv_i^{-1}\subset F_N,\ \ F_n^2\subset F_N,\ \ {\rm and}\ \ \bigcup_{0\le i\le L}F_Nv_i^{-1}\subset F_{N'}\]
with $v_0=e$, replace $\tilde{F}_n$ with $F_{N'}$, and take group elements $g_i\,(i=0,1,\cdots,d)\in Z(F_{N'})\subset Z(F_N)\subset Z(\F_n)$, then we can even get the following stronger result, which is unexpectedly crucial to the calculation of amenability dimension later.

Keeping the same assumption. Then for every $n\ge1$, we can choose another natural number $N\in(n,\infty)$ such that $F_n^2\subset F_N$, and there exists a Rokhlin cover 
\[\mathcal{R}=\{gU_i: g\in F_n, i=0,1,\cdots,L(d+1)\}\]
with shapes $F_n$, and every (closed) base $\overline{U_i}$ being $(F_N,1)$-disjoint, and of course, also $(F_n,1)$-disjoint since $F_n\subset F_n^2\subset F_N$.
\end{cor}

\section{The $C^*$-algebra of $(X,G)$ with $G$ satisfying (WAFC)}\label{sec:5}
\subsection{Finite amenability dimension from the Rokhlin cover}
\begin{lem}[Lemma 7.3, \cite{SWZ}]
Let $G$ be a countable, discrete group, $X$ a locally compact Hausdorff space, and $d\in\N$ a natural number. Let $\af: G\curvearrowright X$ be a free action and $\overline{\af}: G\curvearrowright C_0(X)$ the induced action on the $C^*$-algebra. Then ${\rm dim}_{\rm am}(\overline{\af})\le d$ if and only if for every $\varepsilon>0$, every finite subset $M\Subset G$ and every compact subset $K\subset X$, there exists finitely supported maps $\mu^{(l)}: G\to C_c(X)_{1,+}$ for $l=0,1,\cdots,d$ satisfying:

(a) $\sum_{0\le l\le d}\sum_{g\in G}\mu_g^{(l)}\le {\bf 1}_X$ and $\sum_{0\le l\le d}\sum_{g\in G}\mu_g^{(l)}|_K={\bf 1}_K$;

(b) $\mu_g^{(l)}\mu_h^{(l)}=0$ for all $l=0,1,\cdots,d$ and $g\ne h$ in $G$;

(c) $\mu_h^{(l)}\circ g^{-1}=_{\varepsilon}\mu_{gh}^{(l)}$ for all $l=0,1,\cdots,d$, $g\in M$ and $h\in G$.

\noindent When $X$ is compact, it suffices to check the conditions for $K=X$.
\end{lem}

The following theorem is of a similar fashion to Theorem 7.4 in \cite{SWZ}, except that we will take a “sufficiently nice” Rokhlin cover to complete the proof. This is in fact based on Corollary \ref{am}, which is crucial on a $C^*$-level as we have mentioned.
\begin{thm}\label{fad}
Let $(X, G)$ be a topological dynamical system, where $X$ is a compact metric space and $G$ a countably infinite group with $|Z(F)|=\infty$ for every finite subset $F\subset G$. Let $\{F_n\}\subset G$ be a  symmetric F$\o$lner sequence such that $(G, \{F_n\})$ satisfies (WAFC) with respect to $L_G$.

If the system $(X,G)$ is free and ${\rm dim}(X)=d<\infty$, then the induced $C^*$-algebraic action $\af: G\to {\rm Aut}(C(X))$ has finite amenability dimension. In particular,
\[{\rm dim}^{+1}_{\rm am}(\af)\le L(d+1),\]
where $L=L_G$ is the constant in Definition \ref{AFC}.
\end{thm}

\begin{proof}
Let $\varepsilon>0$ and $M\Subset G$ be a finite subset. Since $\{F_k\}$ is F$\o$lner, we can choose $n\in\N$ such that $F_n$ is $(M,\varepsilon)$-invariant.

By Corollary \ref{am}, we may take $N\in (n,\infty)\cap\N$ such that
\[F_n^2\subset F_N\]
and a Rokhlin cover with shapes $F_n$
\[\mathcal{R}=\{gU_i: i=0,1,\cdots, L(d+1)-1, g\in F_n\}\]
such that $\bigcup_{0\le i\le D}F_nU_i=X$ and every $\overline{U_i}$ is $(F_N,1)$-disjoint. We find a partition of unity subordinate to the open cover $\mathcal{R}$
\[\{\psi_{i,g}: i=0,1,\cdots, L(d+1)-1, g\in F_n\}\]
such that the support of $\psi_{i,g}$ is contained in $gU_i$ for $g\in F_n$ and $\sum_{i,g}\psi_{i,g}=1$. We may also set $\psi_{i,g}=0$ whenever $g\notin F_n$ and $i=0,1,\cdots, L(d+1)-1$.

For every $i=0,1,\cdots, L(d+1)-1$, define $\mu^{(i)}: G\to C(X)_{1,+}$ by
\[\mu^{(i)}_g=\frac{1}{|F_n|}\sum_{h\in F_n}\psi_{i,h^{-1}g}\circ h^{-1}.\]
It is immediate from the definition that each $\mu_g^{(i)}$ is finitely supported and contained in $C(X)_{1,+}$. We now verify that

(a) $\sum_{0\le i\le L(d+1)-1}\sum_{g\in G}\mu_g^{(i)}=1$;

(b) $\mu_{g_1}^{(i)}\mu_{g_2}^{(i)}=0$ whenever $i=0,1,\cdots,L(d+1)-1$ and $g_1\ne g_2$ in $G$;

(c) $\mu_h^{(i)}\circ g^{-1}=_{\varepsilon}\mu_{gh}^{(i)}$.

For (a), we have 
\begin{align*}
\sum_{0\le i\le L(d+1)-1}\sum_{g\in G}\mu_g^{(i)}&=\frac{1}{|F_n|}\sum_{0\le i\le L(d+1)-1}\sum_{g\in G}\sum_{h\in F_n}\psi_{i,h^{-1}g}\circ h^{-1}\\
&=\frac{1}{|F_n|}\sum_{h\in F_n}\left(\sum_{0\le i\le L(d+1)-1}\sum_{g\in G}\psi_{i,h^{-1}g}\right)\circ h^{-1}\\
&=\frac{1}{|F_n|}\sum_{h\in F_n}{\bf 1}_X\circ h^{-1}={\bf 1}_X,
\end{align*}
and hence (a) holds.

For (b), for every $i=0,1,\cdots, L(d+1)-1$ and $g_1\ne g_2\in G$, one sees
\begin{align*}
\mu_{g_1}^{(i)}\mu_{g_2}^{(i)}&=\frac{1}{|F_n|^2}\left(\sum_{h\in F_n}\psi_{i,h^{-1}g_1}\circ h^{-1}\right)\left(\sum_{h\in F_n}\psi_{i,h^{-1}g_2}\circ h^{-1}\right)\\
&=\frac{1}{|F_n|^2}\sum_{h_1,h_2\in F_n}(\psi_{i, h_1^{-1}g_1}\circ h_1^{-1})\cdot (\psi_{i, h_2^{-1}g_2}\circ h_2^{-1}).
\end{align*}
We now prove that for every $h_1, h_2\in F_n$ and $g_1\ne g_2$ in $G$, 
\[(\psi_{i, h_1^{-1}g_1}\circ h_1^{-1})\cdot (\psi_{i, h_2^{-1}g_2}\circ h_2^{-1})=0.\]
In fact, let us assume by contradiction that the product is nonzero. Then there is $x\in X$ such that $\psi_{i, h_1^{-1}g_1}(h_1^{-1}(x))\ne 0$. Since the support of $\psi_{i, h_1^{-1}g_1}$ is contained in $h_1^{-1}g_1U_i$ where $h_1^{-1}g_1\in F_n$, we know that $x\in g_1U_i$ where $g_1\in h_1F_n\subset F_n^2$. Similarly, $x\in g_2U_i$ for $g_2\in h_2F_n\subset F_n^2$. Therefore, $g_1,g_2\in F_n^2$ with
\[g_1U_i\cap g_2U_i\ne\varnothing.\]
But note that by our choice of $N$ and the Rokhlin cover $\mathcal{R}$, $F_n^2\subset F_N$ and the bases $U_i$'s are all $(F_N,1)$-disjoint. This contradicts to $g_1\ne g_2$, which comes (b).

Finally, for (c), this is just a straightforward calculation similar to Theorem 7.4 in \cite{SWZ}:
\begin{align*}
\|\mu_{g'}^{(i)}\circ g^{-1}-\mu^{(i)}_{gg'}\|&=\frac{1}{|F_n|}\left\|\sum_{h\in F_n}\psi_{i,h^{-1}g'}\circ h^{-1}\circ g^{-1}-\sum_{h\in F_n}\psi_{i, h^{-1}gg'}\circ h^{-1}\right\|\\
&=\frac{1}{|F_n|}\left\|\sum_{h\in F_n}\psi_{i,(gh)^{-1}gg'}\circ (gh)^{-1}-\sum_{h\in F_n}\psi_{i, h^{-1}gg'}\circ h^{-1}\right\|\\
&=\frac{1}{|F_n|}\left\|\sum_{h\in gF_n}\psi_{i,h^{-1}gg'}\circ h^{-1}-\sum_{h\in F_n}\psi_{i, h^{-1}gg'}\circ h^{-1}\right\|\\
&=\frac{1}{|F_n|}\left\|\sum_{h\in gF_n\setminus F_n}\psi_{i,h^{-1}gg'}\circ h^{-1}-\sum_{h\in F_n\setminus gF_n}\psi_{i, h^{-1}gg'}\circ h^{-1}\right\|\\
&\le \frac{1}{|F_n|}\left(\sum_{h\in gF_n\setminus F_n}\|\psi_{i,h^{-1}gg'}\circ h^{-1}\|+\sum_{h\in F_n\setminus gF_n}\|\psi_{i, h^{-1}gg'}\circ h^{-1}\|\right)\\
&\le \frac{|F_n\bigtriangleup  gF_n|}{|F_n|}\le\varepsilon({\rm since}\ g\in M\ {\rm and}\ F_n\ {\rm is}\ (M,\varepsilon)-{\rm invariant})\ .
\end{align*}
This completes the proof.
\end{proof}

\subsection{Finite nuclear dimension of $C(X)\rtimes_\af G$ with $G$ satisfying (WAFC)}

Now from Sect. 8 of \cite{GWY}, Theorem 5.14 in \cite{K}, Lemma 8.4 in \cite{SWZ} and Theorem \ref{fad} above, we get the following corollary, regarding other dimensions of $(X, G)$ and its $C^*$-algebra.
\begin{cor}\label{nucdim}
Let $(X, G)$ be a topological dynamical system, where $X$ is a compact metric space and $G$ a countably infinite group with $|Z(F)|=\infty$ for every finite subset $F\subset G$. Let $\{F_n\}\subset G$ be a  symmetric F$\o$lner sequence such that $(G, \{F_n\})$ satisfies (WAFC) with respect to $L_G$.  Denote the induced action $G\curvearrowright C(X)$ by $\af$.

If the system $(X,G)$ is free and ${\rm dim}(X)=d<\infty$, then the dynamic asymptotic dimension ${\rm dad}(X,G)$, tower dimension ${\rm dim}_{\rm tow}$ and fine tower dimension ${\rm dim}_{\rm ftow}$ of $(X, G)$ are all finite. In particular, we have 
\[{\rm dad}^{+1}(X,G)\le L_G(d+1),\ \ {\rm dim}^{+1}_{\rm tow}(X,G)\le{\rm dim}^{+1}_{\rm ftow}(X,G)\le L_G(d+1)^2.\]
Consequently, the crossed product $C(X)\rtimes_{\af}G$ has finite nuclear dimension given by
\[{\rm dim}^{+1}_{\rm nuc}(C(X)\rtimes_{\af}G)\le L_G(d+1)^2.\]
If in addition, the action is minimal, then $C(X)\rtimes_\af G$ is $\mathcal{Z}$-stable, and also classifiable.
\end{cor}

\section{An application of Rokhlin dimension to the embedding problem}\label{sec:6}
\subsection{An embedding result}
Let $(X,G)$ be a topological dynamical system, where $X$ is a topological space and $G$ is a countable amenable group.

Let $\mathcal{U}=(U_i)_{i\in I}$ be a finite open cover of $X$, and $A\Subset G$ a finite subset. Define the finite open cover $\mathcal{U}_A$ by
\[\mathcal{U}_A=\bigvee_{g\in A}g^{-1}(\mathcal{U}).\]
Let $\{F_n\}$ be a F$\o$lner sequence of $G$. Define
\[D(\mathcal{U}, X, G)=\lim_{n\to\infty}\frac{D(\mathcal{U}_{F_n})}{|F_n|}.\]

\begin{rem}
It follows from Proposition 10.2.1 and Theorem 9.4.1 in \cite{MC} that the limit $\lim_{n\to\infty}D(\mathcal{U}_{F_n})/|F_n|$ exists, is finite, and does not depend on the choice of the F$\o$lner sequence $\{F_n\}$.
\end{rem}

\begin{df}\label{md}
The {\it mean (topological) dimension} of the dynamical system $(X,G)$ is the quantity ${\rm mdim}(X,G)\in[0,\infty]$ defined by
\[{\rm mdim}(X,G)=\sup_{\mathcal{U}}D(\mathcal{U}, X, G),\]
where $\mathcal{U}$ runs over all finite open covers of $X$.
\end{df}

In fact, what we need is the following equivalent definition of mean dimension for topological dynamical systems over compact metric spaces. That this is indeed well-defined and equivalent to Definition \ref{md} follows from Theorem 10.4.2 in \cite{MC}. See also Sect. 3 in \cite{GQS} for reference.
\begin{df}[A metric approach to mean dimension]
Let $(X,d)$ be a compact metric space, $G$ a countable amenable group and $(X,G)$ a dynamical system.

For every nonempty finite subset $A\Subset G$, define a metric $d_A$ on $X$ by
\[d_A(x,y)=\max_{g\in A}d(gx, gy)\ \ {\rm for\ all\ }x,y\in X.\]
For every $\varepsilon>0$, define ${\rm widim}_\varepsilon(X, d_A)$ to be the smallest integer $n\ge0$ such that there exists a finite open cover of $X$ with order $n$ and mesh no larger than $\varepsilon$, with respect to the metric $d_A$.

Let $\{F_n\}$ be a F$\o$lner sequence. Then
\[{\rm mdim}(X,G)=\sup_{\varepsilon>0}\lim_{n\to\infty}\frac{{\rm widim}_\varepsilon(X, d_{F_n})}{|F_n|}.\]
\end{df}

We now state our main result.
\begin{thm}\label{ep}
Let $(X, G)$ be a topological dynamical system, where $X$ is a compact metric space and $G$ a countably infinite group with $|Z(F)|=\infty$ for every finite subset $F\subset G$. Let $\{F_n\}\subset G$ be a  symmetric F$\o$lner sequence such that $(G, \{F_n\})$ satisfies (WAFC) with respect to $L_G$.

Let $m\in\N$, $d\in\N\cup\{0\}$ and $(Y, G)$ a factor of $X$. If $(Y, G)$ is free with ${\rm dim}(Y)=d$ and ${\rm mdim}(X, G)<m/2$, then
\[{\rm edim}(X,G)\le(L_G(d+1)+1)m+1\stackrel{\triangle}{=}Q,\]
or in other words, there is an embedding from $(X, G)$ into $(([0,1]^Q)^G, \sigma)$. 

In particular, this applies to $G=\mathbb{H}_{2n+1}(\Z)\times G_1$, where $\mathbb{H}_{2n+1}(\Z)$ is the discrete Heisenberg group, and $G_1$ is any of a finite group, a finitely generated abelian group, a locally finite group or a countable group satisfying (WAFC) with a symmetric F$\o$lner sequence.
\end{thm}

\subsection{The proof of Theorem \ref{ep}}
Recall that for a topological dynamical system $(X,G)$ and any continuous map $f: X\to [0,1]^m$, the induced map $I_f: X\to ([0,1]^m)^G$ is defined by
\[(I_f(x))_g=f(gx).\]
It is immediate that $I_f$ is a continuous map and is equivariant in the sense that, for any $g\in G$ and $x\in X$, $I_f(gx)=\sigma_g(I_f(x))$.

\begin{lem}[Lemma 2.1, \cite{GT1}]\label{3.1}
Let $(X, d)$ be a compact metric space and $f: X\to [0,1]^m$ a continuous map. Suppose that there are $\delta,\varepsilon>0$ satisfying the implication
\[d(x,y)<\varepsilon\Longrightarrow \|f(x)-f(y)\|_\infty<\delta.\]
if ${\rm widim}_\varepsilon(X,d)<m/2$, then there exists an $\varepsilon$-embedding $f': X\to[0,1]^m$ such that
\[\sup_{x\in X}\|f(x)-f'(x)\|_\infty<\delta.\]
\end{lem}

\begin{lem}[Lemma A.5, \cite{Gut1}]\label{A5}
Let $\varepsilon>0$, $m\in\N$ and $B$ be a closed subset of a compact metric space $X$. Let $f': B\to [0,1]^m$ and $\tilde{f}: X\to [0,1]^m$ be two continuous maps such that $\|f'-\tilde{f}|_B\|_\infty<\varepsilon$. Then there exists a continuous map $f: X\to [0,1]^m$ such that
\[f|_B=f'\ \ {\rm and}\ \ \|f-\tilde{f}\|_\infty<\varepsilon.\]
\end{lem}

\begin{lem}\label{eb}
Let $(X, d_X)$ and $(Y, d_Y)$ be compact metric spaces. Let $\varepsilon>0$ be a number. Then for every $\varepsilon$-embedding continuous map $f: X\to Y$, there is $\delta>0$ such that, whenever $g: X\to Y$ is a continuous map with $d(f,g)<\delta$, $g$ is also an $\varepsilon$-embedding, where $d(f,g)=\sup_{x\in X}d_Y(f(x), g(x))$.
\end{lem}
\begin{proof}
First we show that, there is $\delta'>0$ such that for every $y\in Y$, 
\[{\rm diam}(f^{-1}(B_{\delta'}(y)))<\varepsilon.\]
If not, then we can choose sequences $y_n\in Y, x_n, x_n'\in X$ with
\[f(x_n), f(x_n')\in B_{1/n}(y_n)\ \ {\rm and}\ \ d_X(x_n, x_n')\ge\varepsilon.\]
By the compactness of $X$ and $Y$, and also passing to a subsequence, we may assume $x_n\to x, x_n'\to x'$ and $y_n\to y$ as $n\to\infty$. Then $f(x)=f(x')=y$ and $d_X(x,x')\ge\varepsilon$. This contradicts to the assumption that $f$ is an $\varepsilon$-embedding.

Now let $\delta=\delta'/2$ and $g: X\to Y$ be such that $d(f,g)<\delta$. We claim that $g$ is an $\varepsilon$-embedding. In fact, let $x,x'\in X$ so that $g(x)=g(x')$. By the triangle inequality, we have
\begin{align*}
d_Y(f(x), f(x'))&\le d_Y(f(x), g(x))+d_Y(g(x), g(x'))+d_Y(g(x'), f(x'))\\
&<\delta'/2+\delta'/2=\delta',
\end{align*}
and hence $d_X(x,x')<\varepsilon$ by the choice of $\delta'$. This completes the proof.
\end{proof}

We now prove the following Theorem in a similar fashion to Theorem 3.1 in \cite{GQS}.
\begin{thm}
Let $(X, G)$ be a topological dynamical system where $X$ is a compact metric space, and $G$ a countable group acting on $X$.

Let $(Y, G)$ be a factor of $X$ with the factor map $\pi: X\to Y$. If there are $D\in\N\cup\{0\}$ and $L\in\N$ with ${\rm dim}_{\rm Rok}(Y, G)=D$ and ${\rm mdim}(X,G)<L/2$, then the set of maps
\[A=\{f\in C(X, [0,1]^{(D+1)L}): I_f\times\pi: X\to ([0,1]^{(D+1)L})^G\times Y\ {\rm is\ an\ embedding}\}\]
is a dense $G_\delta$ subset of $C(X, [0,1]^{(D+1)L})$. In particular, $A\ne\varnothing$.
\end{thm}

\begin{proof}
For every positive number $\eta>0$, denote
\[A_\eta=\{f\in C(X, [0,1]^{(D+1)L}): I_f\times\pi\ {\rm is\ an\ }\eta-{\rm embedding}\}.\]
It is clear that $A=\bigcap_{n\in\N} A_{1/n}$. Since $C(X, [0,1]^{(D+1)L})$ is a complete metric space with respect to the metric $\|\cdot\|_\infty$, by the Baire's category theorem, it suffices to show that for every $n\in\N$,
\[A_\eta\ {\rm is\ open\ and\ dense\ in}\ C(X, [0,1]^{(D+1)L}).\]
Since every $I_f\times\pi$ is an $\eta$-embedding and $X$ and $([0,1]^{(D+1)L})^G\times Y$ are compact, $A_\eta$ is open by Lemma \ref{eb}. Now it suffices to show that $A_\eta$ is dense.

Now we fix a $\delta>0$. Take $f\in C(X, [0,1]^{(D+1)L})$ arbitrarily. We show that there is a map $g\in A_\eta$ with $\|f(x)-g(x)\|_\infty<\delta$ for all $x\in X$.

We identify $f: X\to [0,1]^{(D+1)L}$ with $f_0\times f_1\times\cdots\times f_D$ where $f_0, f_1,\cdots, f_D: X\to [0,1]^L$ are continuous maps. Since $f_i$ are uniformly continuous, we can choose $\varepsilon\in(0,\eta)$ so that
\begin{align}\label{e1}
d(x,y)<\varepsilon\Longrightarrow \|f_i(x)-f_i(y)\|_\infty<\delta
\end{align}
for all $x,y\in X$ and $i\in\{0,1,\cdots,D\}$. Since ${\rm mdim}(X,G)<L/2$, there is $n\in\N$ with
\[{\rm widim}_\varepsilon(X, d_{F_n})<L|F_n|/2.\]
For each $0\le i\le D$, define the map $F_i: X\to[0,1]^{L|F_n|}=([0,1]^L)^{F_n}$ by
\[(F_i(x))_h=f_i(hx),\ \ h\in F_n.\]
Then by \eqref{e1}, we have the following implication:
\[d_{F_n}(x,y)<\varepsilon\Longrightarrow\|F_i(x)-F_i(y)\|_\infty<\delta\]
for all $x,y\in X$ and $0\le i\le D$. Applying Lemma \ref{3.1}, we can choose $\varepsilon$-embeddings
\[G_0, G_1,\cdots, G_D: X\to [0,1]^{L|F_n|}\]
with respect to the metric $d_{F_n}$ such that $\|F_i(x)-G_i(x)\|_\infty<\delta$ for all $x\in X$. Since ${\rm dim}_{Rok}(Y, G)=D$, by Definition \ref{Rokdim}, we can find open sets $U_0, U_1, \cdots, U_D$ with respect to the integer $n$ such that $\bigcup_{0\le i\le D}F_nU_i=Y$. Define
\[V^h_i=\pi^{-1}(hU_i)\]
for $h\in F_n$ and $0\le i\le D$. Since $\pi$ is surjective, the collection of open sets $\{V_i^h: h\in F_n, 0\le i\le D\}$ forms a Rokhlin cover of $X$, namely,
\[X=\bigcup_{0\le i\le D}\bigsqcup_{h\in F_n}V_i^h.\]
Denote $W_i=\overline{\bigsqcup_{h\in F_n}V_i^h}$ for every $0\le i\le D$. For every $h\in F_n$, we use $p_h: ([0,1]^L)^{F_n}\to [0,1]^L$ to denote the canonical projection onto the $h$-coordinate. For each $0\le i\le D$, define the map $g_i': W_i\to[0,1]^L$ by
\[g_i'(x)=p_h\circ G_i(h^{-1}x)\ \ {\rm for}\ x\in\overline{V_i^h}.\]
Now since $\|F_i(x)-G_i(x)\|_\infty<\delta$ holds for all $x\in X$, for every $0\le i\le D$ and $x\in \overline{V_i^h}$, we have
\[\|f_i(x)-g_i'(x)\|_\infty=\|p_h\circ F_i(h^{-1}x)-p_h\circ G_i(h^{-1}x)\|<\delta,\]
that is, $\|f_i(x)-g_i'(x)\|<\delta$ on the domain $W_i$ of $g_i'$. Since $W_i$ is a closed subset of $X$, from Lemma \ref{A5}, we can then take a continuous extension $g_i: X\to [0,1]^L$ of $g_i'$ so that $\|f_i(x)-g_i(x)\|_\infty<\delta$ for all $x\in X$.

Now we define 
\[g=g_0\times g_1\times\cdots\times g_D: X\to [0,1]^{(D+1)L}.\]
It is clear that, from the choice of $g_i$, we have $\|f(x)-g(x)\|_\infty<\delta$ for all $x\in X$. Now it suffices to show that $g\in A_\eta$.

For this, we assume that $I_g\times\pi(x)=I_g\times\pi(y)$ for $x,y\in X$. We then have $\pi(x)=\pi(y)$ obviously. By the definition of the cover $\{V_i^h: i=0,1,\cdots, D, h\in F_n\}$, we know that 
\[x,y\in V_i^h\]
for some $i\in\{0,1,\cdots, D\}$  and $h\in F_n$. On the other hand, from $I_g(x)=I_g(y)$, we see that $g_i(wx)=g_i(wy)$ for all $w\in G$. Note that if $w\in F_n$, then $wh^{-1}x\in V_i^w$. In fact, this is because
\[wh^{-1}x\in wh^{-1}(V_i^h)=wh^{-1}\pi^{-1}(hU_i)\stackrel{\pi\ {\rm is\ a\ factor}}{=}\pi^{-1}(wh^{-1}hU_i)=V_i^w.\]
This follows that
\[g_i(wh^{-1}x)=g_i'(wh^{-1}x)=p_w\circ G_i(w^{-1}wh^{-1}x)=p_w\circ G_i(h^{-1}x).\]
Similarly, we have $g_i(wh^{-1}y)=p_w\circ G_i(h^{-1}y)$. Since $w\in F_n$ is arbitrary, this implies
\[G_i(h^{-1}x)=G_i(h^{-1}y).\]
However, by the choice of $G_i$ which is an $\varepsilon$-embedding with respect to the metric $d_{F_n}$, we conclude that
\[d(x,y)=d(hh^{-1}x, hh^{-1}y)\le d_{F_n}(h^{-1}x, h^{-1}y)<\varepsilon<\eta\]
as designed. This completes the proof.
\end{proof}

Note that for any countably infinite group $G$ and a compact metric space $Y$ on which there is a group action of $G$, if the dynamical system $(Y, G)$ is free and ${\rm dim}(Y)<\infty$, then $(Y, G)$ embeds into $([0,1]^G, \sigma)$. In fact, the proof goes almost the same to Theorem 4.2 in \cite{Ja} except that we replace the $T^i(y)\,(i\in\Z)$ with $gy$ for some $g\in G$, we will not give any further details. Also see Theorem 8.3.1 in \cite{MC}. Therefore, applying Corollary \ref{frd} and Corollary \ref{hsb}, we deduce the following.

\begin{cor}[$=$Theorem \ref{ep}]
Let $(X, G)$ be a topological dynamical system, where $X$ is a compact metric space and $G$ a countably infinite group with $|Z(F)|=\infty$ for every finite subset $F\subset G$. Let $\{F_n\}\subset G$ be a  symmetric F$\o$lner sequence such that $(G, \{F_n\})$ satisfies (WAFC) with respect to $L_G$.

Let $m\in\N$, $d\in\N\cup\{0\}$ and $(Y, G)$ a factor of $X$. If $(Y, G)$ is free with ${\rm dim}(Y)=d$ and ${\rm mdim}(X, G)<m/2$, then there is an embedding from $(X, G)$ into $(([0,1]^Q)^G, \sigma)$, where
\[Q=\left(L_G(d+1)+1\right)m+1.\]
\end{cor}

\section{Appendix A. Outlook and questions}\label{sec:7}

\subsection{Strong Rokhlin dimension and global markers}
Enlightened by the observation in Corollary \ref{am}, we define the following concept.

\begin{df}\label{srd}
Let $(G, \{F_n\})$ be an amenable countable group acting on a compact metric space $X$. 

Let $D\in\N\cup\{0\}$. We say that $(X,G)$ has {\it strong (topological) Rokhlin dimension $D$}, and denote ${\rm dim}_{\rm sRok}(X,G)=D$, if $D$ is the smallest number with the following property: 

For every $n\in\N$, there exist open subsets $U_0, U_1, \cdots, U_D\subset X$ and an $n'\ge n$ such that

(i) $F_n^2\subset F_{n'}$;

(ii) For every $0\le i\le D$, $\overline{U_i}$ is $(F_{n'},1)$-disjoint;

(iii) The towers $\{(U_i, F_n)\}_{0\le i\le D}$ covers $X$, that is, $\bigcup_{0\le i\le D}F_nU_i=X$.
\end{df}

\begin{cor}\label{eqf}
Let $(G, \{F_n\})$ be an amenable countable group acting on a compact metric space $X$ of finite covering dimension. Consider the following conditions.

\noindent $(i)$ $(G,X)$ is free, and

\noindent $(ii)$ $(G,X)$ has finite strong Rokhlin dimension,

\noindent then $(ii)\Rightarrow(i)$. Furthermore, if $G$ is a countably infinite group with the property (I), and in addition, $(G, \{F_n\})$ satisfying (WAFC) with $\{F_n^{-1}\}$ being approximate groups with a uniform constant $L_G$, then $(i)\Leftrightarrow(ii)$. 
\end{cor}

\begin{proof}
With Corollary \ref{am}, all that remains to check is (ii) implies (i) for any amenable group. Suppose to the contrary that this is not the case. Then there is $x\in X$ and a group element $g\in G\setminus\{e\}$ with $gx=x$. Choose an $n\in\N$ such that $g\in F_n$.

Since ${\rm dim}_{\rm sRok}=D<\infty$, there are open sets $U_0, U_1,\cdots, U_D\subset X$ and a natural number $n'\ge n$ such that there exists a Rokhlin cover
\[\mathcal{R}=\{(U_i, F_n): i=0,1,\cdots,D\}\]
with $F_n^2\subset F_{n'}$ and every $\overline{U_i}$ being $(F_{n'},1)$-disjoint. Assume that $x\in hU_i$ for some $h\in F_n$ and $0\le i\le D$. Then $hU_i\ni x=gx\in ghU_i$. This means $x\in hU_i\cap ghU_i\ne\varnothing$. However, note that $h\in F_n\subset F_n^2\subset F_{n'}$ and $gh\in F_n^2\subset F_{n'}$, which contradicts to the assumption that $U_i\subset \overline{U_i}$ is $(F_{n'}, 1)$-disjoint.
\end{proof}

\begin{rem}
Checking the proof of Theorem \ref{fad}, one observes that, the quantity that we are in fact using is the finite strong Rokhlin dimension, and since from Corollary \ref{eqf}, for the class of dynamical systems that we are considering, the freeness is equivalent to having a finite strong Rokhlin dimension, according to which we believe that, the strong Rokhlin dimension is actually the right concept for actions of general groups on compact metric spaces.
\end{rem}

\begin{qn}
For what class of countable amenable groups that one can deduce the finiteness of strong (topological) Rokhlin dimension from the freeness and the finiteness of covering dimension of the spaces on which the groups act on?
\end{qn}

\begin{qn}[=Question \ref{fgg}]
Does every finitely generated subgroup of a group satisfying (WAFC) necessarily have subexponential growth?
\end{qn}

For $G=\Z$ and $n\ge1$, it is obvious that if $O$ is an $n$-marker, then so is $T^i(O)$ for any $i\in\Z$. This is the case for any abelian group. However, for non-abelian one, this is no longer ensured. Therefore, we make the following definition.

\begin{df}
Let $(X,G)$ be a topological dynamical system, $F\Subset G$ a finite subset and $O\subset X$ an open set. We say $O$ is a {\it global $F$-marker}, if $gO$ is an $F$-marker for all $g\in G$.

Furthermore, suppose that $G$ is a countable amenable group with $\{F_n\}$ being its F$\o$lner sequence. We say $G$ has the {\it global marker property}, if for every $n\ge1$, $X$ admits a global $F_n$-marker $O$.
\end{df}
\begin{qn}
Let $(X,G)$ be a free topological dynamical system with the marker property where $G$ is a countable amenable group. Then does it necessarily have the global marker property? 
\end{qn}

\subsection{Other Rokhlin-type properties and their relations}
\begin{df}[Definition 3.5, \cite{KS}]
A free action $G\curvearrowright X$ on a compact metric space is said to be {\it almost finite in measure} if for every finite set $K\Subset G$ and $\delta,\varepsilon>0$, there is a (disjoint) open castle $\{(V_i,S_i)\}_{i\in I}$ with levels of diameter less than $\delta$ such that

(i) Every shape $S_i$ is $(K,\delta)$-invariant;

(ii) $\underline{D}(\bigsqcup_{i\in I}S_iV_i)\ge1-\varepsilon$,

\noindent where $\underline{D}(A)=\sup_{F}\inf_{x\in X}(1/|F|)\sum_{s\in F}1_A(sx)$ for any subset $A\subset X$, called the {\it lower Banach density of $A$}.
\end{df}

\begin{df}[Definition 1.9, \cite{Gut2}]
We say that a free dynamical system $(X,T)$ has the {\it topological Rokhlin property}, if for every $\varepsilon>0$, there exists a continuous function $f: X\to\R$ so that the {\it exceptional set}
\[E_f=\{x\in X: f(Tx)\ne f(x)+1\}\]
has orbit capacity less then $\varepsilon$, that is, ${\rm ocap}(E_f)<\varepsilon$.
\end{df}

\begin{df}[Definition 5.1, \cite{Gut3}]
We say that a free dynamical system $(X,T)$ has the {\it strong topological Rokhlin property}, if for every $n\in\N$, there exists a continuous function $f: X\to\R$ so that the sets $T^{-i}(E_f), i=0,1,\cdots,n-1$ are pairwise disjoint.
\end{df}

\begin{df}[Definition 3.1, \cite{N}]
A topological dynamical system $(X,\Gamma)$, where $\Gamma$ is a discrete amenable group, is said to have the {\it uniform Rokhlin property} if for any $\varepsilon>0$ and any finite set $K\Subset \Gamma$, there exists open sets $B_1, B_2,\cdots,B_S\subset X$ and $(K,\varepsilon)$-invariant sets $\Gamma_1,\Gamma_2, \cdots,\Gamma_S\subset\Gamma$ such that
\[B_s\gamma: \gamma\in\Gamma_s, s=1,2,\cdots,S\]
are mutually disjoint and 
\[{\rm ocap}(X\setminus \bigsqcup_{1\le s\le S}\bigsqcup_{\gamma\in\Gamma_s}B_s\gamma)<\varepsilon.\]
\end{df}

\begin{prop}\label{im}
Let $(X, T)$ be a free topological dynamical system on a compact metric space $X$. Then we have the following implications.
\[AF_m\stackrel{(a)}{\Longleftrightarrow}SBP\stackrel{(b)}{\Longrightarrow}MP\stackrel{(c)}{\Longleftrightarrow}STRP\stackrel{(d)}{\Longrightarrow}TRP\stackrel{(e)}{\Longrightarrow}URP,\]
where $AF_m$=almost finite in meausre, SBP=small boundary property, MP=marker property, STRP=strong topological Rokhlin property, TRP=topological Rokhlin property and URP=uniform Rokhlin property.
\end{prop}

\begin{proof}
The implication (a) follows from Theorem 5.6 in \cite{KS}(which in fact concludes that (a) holds for all discrete amenable groups), while (b) comes from the discussion below Conjecture 7.5 in \cite{TTY}. According to Theorem 7.3 in \cite{Gut1}, we get the equivalence (c). The implication (d) is trivial. It suffices to show (e) now.

Suppose that $(X,T)$ has TRP. Applying Lemma 4.11 in \cite{Gut2}, for every $\varepsilon>0$, we can choose a continuous function $f: X\to\R$ such that if we define
\[\tilde{E}_f=\{x\in X: (f(x)\notin\Z)\vee(f(Tx)\ne f(x)+1)\},\]
then ${\rm ocap}(\tilde{E}_f)<\varepsilon$. Note that in the proof of Lemma 3.6 in \cite{N}, one needs the assumption of $(X, T)$ being an extension of a free minimal system only for the existence of such a continuous function $f$. Therefore, we will not give the proof repeatedly, and conclude that the implication (e) then follows from Lemma 3.6 in \cite{N}.
\end{proof}

\begin{rem}
Note that there is only one gap between $AF_m$ and $URP$, namely, in the definition of almost finiteness in measure, all the levels are further required to have arbitrarily small diameters. From Proposition \ref{im}, we see immediately that URP does not imply $AF_m$ in general, as we can choose a minimal dynamical system $(X,T)$ with strictly positive mean dimension. It is well known that such minimal system exists, and that minimality always follows the marker property, hence implies URP by Proposition \ref{im}. However, since it has positive mean dimension, it cannot have SBP, and consequently fails to have $AF_m$.
\end{rem}

\vspace{1cm}

\noindent Sihan Wei, School of Mathematics and Statistics, university of Glasgow, Glasgow, UK

{\em Email address}: {\bf sihan.wei@glasgow.ac.uk}
\quad\par
\quad\par
\noindent Zhuofeng He, BIMSA, Beijing, China

{\em Email address}: {\bf zhuofenghe@bimsa.cn}


\begin{thebibliography}{99}
\addcontentsline{toc}{chapter}{References}
\thispagestyle{plain}



\bibitem{AGG} F. Abadie, E. Gardella, S. Geffen, {\em Partial $C^*$-dynamics and Rokhlin dimension}. Ergodic Theory Dynam. Systems. {\bf 42} (2022), no. 10, 2991--3024.







\bibitem{BGT} E. Breuillard, B. Green, T. Tao, {\em The structure of approximate groups}. Publ. Math. Inst. Hautes $\acute{\rm E}$tudes Sci. {\bf 116} (2012), 115--221.



\bibitem{MC} Michel Coornaert, {\em Topologcal dimension and dynamical systems}. Universitext. Springer Cham, France, 2015.







\bibitem{CHT} M. Cordes, T. Hartnick, Vera Toni$\acute{\rm c}$, {\em Foundations of geometric approximate group theory}. Preprint, 2022.  \href{https://arxiv.org/abs/2012.15303}{arXiv: 2012.15303v2}




\bibitem{EN} G. Elliott, Z. Niu, {\em The $C^*$-algebra of a minimal homeomorphism of zero mean dimension}. Duke. Math. J. (18) {\bf 166} (2017), 3569--3594.




\bibitem{GHS} E. Gardella, I. Hirshberg, L. Santiago, {\em Rokhlin dimension: duality, tracial properties, and crossed products}. Ergodic Theory Dynam. Systems. {\bf 41} (2012), no. 2, 408--460.



\bibitem{Gr} R. Grigorchuk, {\em Degrees of growth of finitely generated groups and the theory of invariant means}. Izv. Akad. Nauk SSSR Ser. Mat. (5) {\bf 48} (1984), 939--985.






\bibitem{G} M. Gromov, {\em Groups of polynomial growth and expanding maps}. Inst. Hautes $\acute{\rm E}$tudes Sci. Publ. Math. {\bf 53} (1981), 53--73.




\bibitem{GWY} E. Guentner, R. Willett, G. Yu, {\em Dynamic asymptotic dimension: relation to dynamics, topology, coarse geometry, and $C^*$-algebras}. Math. Annalen. {\bf 367} (2017), 785--829.



\bibitem{Gut2} Y. Gutman, {\em Embedding $\Z^k$-actions in cubical shifts and $\Z^k$-symbolic extensions}. Ergodic Theory Dynam. Systems. {\bf 31} (2011), 383--403.




\bibitem{Gut1} Y. Gutman, {\em Mean dimension and Jaworski-type theorems}. PRok. Lond. Math. Soc. (3) {\bf 111} (2015), no. 4, 831--850.



\bibitem{Gut3} Y. Gutman, {\em Embedding topological dynamical systems with period points in cubical shifts}. Ergodic Theory Dynam. Systems. {\bf 37} (2015), no. 2, 512--538.




\bibitem{GQS} Y. Gutman, Y. Qiao, G. Szab$\acute{\rm o}$, {\em The embedding problem in topological dynamics and Takens' theorem}. Nonlinearity. {\bf 31} (2018), no. 2, 597--620.



\bibitem{GT1} Y. Gutman, M. Tsukamoto, {\em Mean dimension and a sharp embedding theorem: extensions of aperiodic subshifts}. Ergodic Theory Dynam. Systems. {\bf 34} (2014), no. 6, 1888--1896.


\bibitem{GT2} Y. Gutman, M. Tsukamoto, {\em Embedding minimal dynamical systems into Hilbert cubes}. Invent. math. {\bf 221} (2020), 113--166.



\bibitem{HWZ} I. Hirshberg, W. Winter, J. Zacharias, {\em Rokhlin dimension and $C^*$-dynamics}. Comm. Math. Phys. {\bf 335} (2015), no. 2, 637--670.



\bibitem{Ja} D. Janzen, {\em F$\o$lner nets for semidirect products of amenable groups}. Canad. Math. Bull. {\bf 51} (1) (2008), 60--66.



\bibitem{Ja} A. Jaworski, {\em The Kakutani-Beboutov theorem for groups}. Ph. D. dissertation. University of Maryland, 1974.



\bibitem{Jo} M. Joseph, {\em Amenable wreath products with non almost finite actions of mean dimension zero}. Preprint, 2023. \href{https://arxiv.org/abs/2301.07616}{arXiv:2301.07616v2}.



\bibitem{K} D. Kerr, {\em Dimension, comparison, and almost finiteness}. J. Eur. Math. Soc. {\bf 22} (2020), no. 11, 3697--3745.








\bibitem{KS} D. Kerr, G. Szab$\acute{\rm o}$, {\em Almost finiteness and the small boundary property}. Comm. Math. Phys. {\bf 374} (2020), no. 1, 1--31.



\bibitem{L1} H. Lin, {\em AF-embeddings of the crossed products of AH-algebras by finitely generated abelian groups}. Int. Math. Res. Pap. IMRP (2008), no. 3, Art. ID rpn007, 67.



\bibitem{Ln} E. Lindenstrauss, {\em Mean dimension, small entropy factors and an embedding theorem}. Inst. Hautes $\acute{\rm E}$tudes Sci. Publ. Math. {\bf 89}(1) (1999), 227--262.



\bibitem{LW} E. Lindenstrauss and B. Weiss, {\em Mean topological dimension}. Israel J. Math. {\bf 115} (2000), 1--24.



\bibitem{N} Z. Niu, {\em Comparison radius and mean topological dimension: Rokhlin property, comparison of open sets, and subhomogeneous $C^*$-algebras}. J. Anal. Math. {\bf 146} (2020), no. 2, 595--672.





\bibitem{Sza} G. Szab$\acute{\rm o}$, {\em The Rokhlin dimension of topological $\Z^m$-actions}. PRok. Lond. Math. Soc. (3) {\bf 110} (2015), no. 3, 673--694.


\bibitem{Sza2} G. Szab$\acute{\rm o}$, {\em Rokhlin dimension: absorption of model actions}. Anal. PDE. {\bf 12} (2019), no. 5, 1357--1396.



\bibitem{SWZ} G. Szab$\acute{\rm o}$, J. Wu and J. Zacharias, {\em Rokhlin dimension for actions of residually finite groups}. Ergodic Theory Dynam. Systems {\bf 39} (2019), no. 8, 2248--2304.

\bibitem{T} T. Tao, {\em Product set estimates for non-commutative groups}. Combinatorica. (5) {\bf 28} (2008), 547--594.



\bibitem{TTY} M. Tsukamoto, M. Tsutaya, M. Yoshinaga, {\em $G$-index, topological dynamics and marker property}. Israel J. Math. {\bf 251} (2020), no. 2, 737--764.



\bibitem{WZ} W. Winter and J. Zacharias, {\em The nuclear dimension of $C^*$-algebras}. Adv. Math. {\bf 224} (2010), 461--498.


\end{thebibliography}
\end{document}